\documentclass[11pt,a4paper,final]{amsart}
\usepackage[a4paper,left=30mm,right=30mm,top=30mm,bottom=30mm,marginpar=23mm]{geometry} 
\usepackage{amsmath}
\usepackage{amssymb}
\usepackage{amsthm}
\usepackage{amscd}
\usepackage[ansinew]{inputenc}
\usepackage{cite}
\usepackage{bbm}
\usepackage{color}
\usepackage[english=american]{csquotes}
\usepackage[final]{graphicx}
\usepackage{hyperref}
\usepackage{calc}
\usepackage{mathptmx}

\newtheorem{innercustomthm}{Lemma}
\newenvironment{customthm}[1]
  {\renewcommand\theinnercustomthm{#1}\innercustomthm}
  {\endinnercustomthm}

\numberwithin{equation}{section}

\newtheoremstyle{thmlemcorr}{10pt}{10pt}{\itshape}{}{\bfseries}{.}{10pt}{{\thmname{#1}\thmnumber{ #2}\thmnote{ (#3)}}}
\newtheoremstyle{thmlemcorr*}{10pt}{10pt}{\itshape}{}{\bfseries}{.}\newline{{\thmname{#1}\thmnumber{ #2}\thmnote{ (#3)}}}
\newtheoremstyle{remexample}{10pt}{10pt}{}{}{\bfseries}{.}{10pt}{{\thmname{#1}\thmnumber{ #2}\thmnote{ (#3)}}}

\theoremstyle{thmlemcorr}
\newtheorem{theorem}{Theorem}
\numberwithin{theorem}{section}
\newtheorem{lemma}[theorem]{Lemma}

\newtheorem{proposition}[theorem]{Proposition}

\newtheorem{definition}[theorem]{Definition}

\theoremstyle{thmlemcorr*}
\newtheorem{theorem*}{Theorem}
\newtheorem{lemma*}[theorem]{Lemma}
\newtheorem{corollary*}[theorem]{Corollary}
\newtheorem{proposition*}[theorem]{Proposition}
\newtheorem{problem*}[theorem]{Problem}
\newtheorem{conjecture*}[theorem]{Conjecture}
\newtheorem{definition*}[theorem]{Definition}

\theoremstyle{remexample}
\newtheorem{remark}[theorem]{Remark}

\newcommand{\Acal}{\mathcal{A}}

\newcommand{\Rcal}{\mathcal{R}}

\newcommand{\R}{\mathbb{R}}
\newcommand{\N}{\mathbb{N}}

\newcommand{\Ucal}{\mathcal{U}}

\newcommand{\setb}[2]{\bigl\{\, #1 \ \ \textup{\textbf{:}}\ \ #2 \,\bigr\}}

\newcommand{\Irm}{\mathrm{I}}

\renewcommand\appendix{\par
        \renewcommand\thesection{A}
        \renewcommand\thesubsection{A\arabic{subsection}}
        \renewcommand\thetable{A\arabic{table}}}

\def\XXint#1#2#3{{\setbox0=\hbox{$#1{#2#3}{\int}$} 
\vcenter{\hbox{$#2#3$}}\kern-.5\wd0}}

\newcommand{\W}{\mathrm{W}}
\newcommand{\LL}{\mathrm{L}}
\newcommand{\st}{\ensuremath{\mathrm{ \hspace*{1mm}\textup{\textbf{:}}\hspace*{1mm} }}}
\newcommand{\dx}{\, \mathrm{d}x}
\newcommand{\dy}{\, \mathrm{d}y}
\newcommand{\dt}{\, \mathrm{d}t}
\newcommand{\dv}[1]{\, \mathrm{d}#1}

\newcommand{\eps}{\varepsilon}
\renewcommand{\phi}{\varphi}

\newcommand{\varA}{\mathrm{Var}(\mathcal A)}

\begin{document}


\title[Conditions for strong local minimisers on non-smooth domains]{Necessary and sufficient conditions for the strong local minimality of ${\rm C}^1$ extremals on a class of non-smooth domains}

\author{Judith Campos Cordero}
\address{Department of Mathematics, Universidad Aut\'{o}noma Metropolitana--Iztapalapa,
Av. San Rafael Atlixco 186, 09340 Mexico City, Mexico}
\email{judith@ciencias.unam.mx}

\author{Konstantinos Koumatos}
\address{Department of Mathematics, University of Sussex, Pevensey 2 Building, Falmer, Brighton, BN1 9QH, UK}
\email{k.koumatos@sussex.ac.uk}


\hypersetup{
  pdfauthor = {},
  pdftitle = {...},
  pdfsubject = {},
  pdfkeywords = {}
}


\maketitle


\begin{abstract}

\noindent\textsc{}
Motivated by applications in materials science, a set of quasiconvexity at the boundary conditions is introduced for domains that are locally diffeomorphic to cones. These conditions are shown to be necessary for strong local minimisers in the vectorial Calculus of Variations and a quasiconvexity-based sufficiency theorem is established for ${\rm C}^1$ extremals defined on this class of non-smooth domains. The sufficiency result presented here thus extends the seminal theorem by Grabovsky \& Mengesha (2009), where smoothness assumptions are made on the boundary.\vspace{0.1cm}

\vspace{0.1cm}

\vspace{0.1cm}

\noindent\textsc{Date:} \today{}
\end{abstract}

\setcounter{tocdepth}{1} 

\tableofcontents

\section{Introduction}

The problem of finding necessary and sufficient conditions for a map to be a strong local minimiser of a given functional is a fundamental question in the Calculus of Variations. The scalar case was first solved by Weierstrass and, later, Hestenes  \cite{HestenesSuff} extended these results by considering minimisers that allow more than one independent variable. 

Regarding the vectorial case, the problem has received considerable attention over the past decades. In \cite{zhang92QCremarks} Zhang proved a sufficiency theorem in which the strong quasiconvexity of the integrand is used to prove that smooth solutions to the weak Euler-Lagrange equations are spatially-local minimisers. In other words, this means that minimality can be obtained when considering sufficiently  small domains. 

Ball conjectured in \cite{BallCoVMatSci} that a natural way to extend the Weierstrass theory to the vectorial case had to be based on the conditions of quasiconvexity in the interior and at the boundary, the latter having been introduced in \cite{BallMarsden} and shown to be a necessary condition for strong local minima; see also \cite{kruzikBQC,kruzikBNL,mielkeBQC} for other works related to the notion of quasiconvexity at the boundary. More precisely, Ball's conjecture states that under suitable quasiconvexity conditions, if a map is sufficiently smooth, satisfies the weak Euler-Lagrange equations and the strict positivity of the second variation, then it is a strong local minimiser. 

Further achievements for the vectorial case were obtained in  \cite{TaheriSufficiency}. In addition, Kristensen and Taheri showed that, under strong quasiconvexity and $p$-growth conditions, $\W^{1,p}$-local minimisers are \textit{partially regular}, i.e. of class ${\rm C}^{1,\alpha}$ almost everywhere in their domain. Their result extended the partial regularity that Evans had established for the first time for global minimisers under similar assumptions on the integrand \cite{Evanspr}. Furthermore, and in contrast to the previous achievements on the vectorial problem, Kristensen and Taheri modified the remarkable example by M\"uller and {\v{S}}ver{\'a}k in \cite{MyS} to construct a strongly quasiconvex integrand with strictly positive second variation, such that the corresponding weak Euler-Lagrange equations admit Lipschitz solutions that are not $\W^{1,p}$-local minimisers \cite{KristensenTaheri}. This made it clear that Lipschitz regularity of the extremal is not enough to ensure strong local minimality. Sz{\'e}kelyhidi extended further the example of Kristensen and Taheri to the case of polyconvex integrands \cite{Szekelyhidi}.

However, it was only until the seminal paper \cite{GM09} of Grabovsky and Mengesha that the conjecture of Ball was settled. They showed that, for domains of class ${\rm C}^1$, a strong version of the aforementioned quasiconvexity conditions is in fact sufficient for a ${\rm C}^1$ map to be a strong local minimiser of an integral functional under specific growth and coercivity assumptions on the integrand. 

The aim of the current work is to address the vectorial problem without the ${\rm C}^1$ regularity assumption on the domain but instead allowing certain types of singularities of the boundary. More specifically, we establish a quasiconvexity-based sufficiency result for minimisers defined on domains that are locally diffeomorphic to cones (see Definition \ref{def:domains} for a precise definition). 

Similarly to the work in \cite{GM09}, it is shown that when complemented with the satisfaction of the weak Euler-Lagrange equations and the (strong) positivity of the second variation, a strong version of the quasiconvexity in the interior and the suitably adapted conditions of quasiconvexity at the singular boundary are indeed sufficient for a ${\rm C}^1$ map to be a strong local minimiser. 

The need to extend the Weierstrass theory to domains of this type arises naturally in view of applications. For example, models based on energy minimisation, and consequently the techniques of the vectorial Calculus of Variations, have been very successful in materials science where typical specimens are polyhedral. Indeed, the work presented here has been largely motivated by \cite{BKARMA} where, in a simplified model, a set of quasiconvexity conditions at edges and corners of a (convex) polyhedral domain was employed to explain remarkable experimental observations in a shape-memory alloy (see \cite{icomat11}). In particular, it was shown that, in this simplified model, the quasiconvexity conditions hold in the interior and at edges, thus preventing the localised nucleation of the high temperature phase; in contrast, and consistent with observations, quasiconvexity was lost at certain corners allowing for nucleation. The sufficiency result presented in the current work could be used to strengthen the result of \cite{BKARMA} in a different modelling regime where variations that are not localised are also considered. This may indeed be an interesting direction to pursue. We remark that quasiconvexity conditions at points of the boundary with conical singularities were first discussed in \cite[Remark 2]{BallMarsden} but, to the best of the authors' knowledge, this idea had not been previously applied. 

The plan of the paper is as follows: in Section \ref{sec:definitions}, we state the general assumptions under which we establish our results. We further give all the required definitions on our domains as well as the generalised quasiconvexity conditions at the singular boundary. In Section \ref{sec:necessity}, we show that all these quasiconvexity at the boundary conditions are indeed necessary for a strong local minimiser and, in Section \ref{sec:sufficiency}, we state and prove the corresponding quasiconvexity-based sufficiency theorem which generalises the result of \cite{GM09} to the case of domains locally diffeomorphic to cones. Our strategy of proof shares many of the underlying ideas found in \cite{GM09}. However, in the present work we establish a G{\aa}rding inequality, that encapsulates the quasiconvexity conditions, and is crucially used to show that the concentrating part of a sequence of suitable variations does not lower the energy. On the other hand, in \cite{GM09} the behaviour of the energy on  this concentrating part of the variations is handled via a localization argument. 

Our proof is broadly motivated by the work in \cite{JCCVMO}, and the resulting G{\aa}rding inequality comes as a generalisation of Zhang's result on spatially-local minimisers. The proof of this generalisation, in the general case under consideration, becomes notationally involved and potentially difficult to follow its otherwise simple and clear ideas. Hence, for the convenience of the reader, the lemma is proved in the simpler case of the domain being itself a cone, though the result is stated for the general case. The interested reader is referred to Section~\ref{sec:AppendDiffeom} for a full proof of the lemma.


\section{Preliminaries \& Definitions}\label{sec:definitions}

Let $\Omega\subset\R^{d}$ be a bounded Lipschitz domain (i.e. open and connected) and consider an integral functional of the form
\begin{equation}\label{eq:functional_intro}
I(u) = \int_{\Omega}F(x,u(x),\nabla u(x))\dx,
\end{equation}
which is to be minimised over a set of admissible maps $\Acal$, where $u\colon\Omega\to\R^{N}$ and the function $F\colon\overline{\Omega}\times\R^N\times\R^{N\times d}\to\R$ are continuous. The Weierstrass problem consists in finding necessary and sufficient conditions for a map $u_0$ to be a strong local minimiser of the functional $I$ in $\Acal$, the space of admissible maps $\Acal$ being part of the problem.

For our purposes, we define
\begin{equation}\label{eq:admissible_maps}
\Acal:=\left\{u\in {\rm C}^1({\overline{\Omega}},\R^N) \st u(x)=\bar{u}(x)\mbox{ for all }x\in\Gamma_D\right\},
\end{equation}
where $\Gamma_D\subseteq \partial\Omega$ and $\bar u$ is continuously differentiable on some open set in $\R^d$ containing $\overline{\Gamma_D}$.

Note that if $u\in\Acal$, then $u(x)=\bar u(x)$ for all $x\in\overline{\Gamma_D}$. Hence, without loss of generality, we assume that $\Gamma_D$ is the interior of $\overline{\Gamma_D}$, relative to $\partial\Omega$. By defining $\Gamma_N:=\partial\Omega\setminus\overline{\Gamma_D}$, $\Gamma_N$ is a relatively open subset of $\partial\Omega$ and $\partial\Omega=\Gamma_D\cup\overline{\Gamma_N}$. Indeed, if $x\in \partial\Omega\setminus\overline{\Gamma_N}$, then $x$ has an open neighbourhood in $\partial\Omega$ that does not intersect $\Gamma_N$. Therefore, this neighbourhood must belong to the interior of $\overline{\Gamma_D}$, which is $\Gamma_D$. 
Additionally, we must require that $\Gamma_D$ has itself a Lipschitz boundary in $\partial\Omega$ in the sense of \cite[Definition 2.1]{Virginia}. This assumption allows us to make the identification (see \cite[Proposition 6.2]{Virginia})
\begin{align*}
&\overline{\{u\in C^\infty(\overline{\Omega},\R^N) \st u(x) = 0\mbox{ on }\Gamma_D\}^{\W^{1,p}}} \\
 =&\{u\in \W^{1,p}(\Omega,\R^N) \st Tu(x) = 0\mbox{ $\mathcal{H}^{d-1}$-  a.e. on }\Gamma_D\},
\end{align*}
where $T$ denotes the trace operator. 

We note that the subscripts $D$ and $N$ in $\Gamma_D$ and $\Gamma_N$ stand for Dirichlet and Neumann and are meant to help the reader associate the two parts of the boundary of $\Omega$ as the prescribed ($\Gamma_D$) and free ($\Gamma_N$) boundary, respectively.

\begin{definition}\label{def:slm_standard}
A map $u_0\in\Acal$ is a strong local minimiser of $I$ in $\Acal$ if there exists an $\varepsilon > 0$ such that, whenever $u\in\Acal$ and $||u - u_0||_\infty<\varepsilon$, it holds that $I(u)\geq I(u_0)$.
\end{definition}

Equivalently, because the uniform topology on the space of continuous functions is metrizable, the notion of strong local minimisers can be expressed in terms of sequences.
\begin{definition}
Let the space of variations $\varA$ be given by
\begin{equation}\label{eq:varA}
\varA = \left\{
\varphi\in {\rm C}^1(\overline{\Omega},\R^N)\st \varphi(x)=0 \mbox{ for all }x\mbox{ in }\Gamma_D
\right\}.
\end{equation}
We say that a sequence $\{\varphi_j\}\subset\varA$ is a strong variation if $\varphi_j\to 0$, as $j\to\infty$, uniformly in $x\in\Omega$. 
\end{definition}
Then, a map $u_0\in\Acal$ is a strong local minimiser of $I$ in $\Acal$ if, and only if, for every strong variation $\{\varphi_j\}$ there exists $J>0$ such that
\begin{equation}\label{eq:slm_variations}
I(u_0+\varphi_j)\geq I(u_0)\quad\mbox{for all }j\geq J.
\end{equation}

More generally, given an open set $\omega\subseteq\R^d$ such that $\omega\cap\Omega\neq\emptyset$, we consider the following space of variations defined in $\omega$.
\begin{equation*}
\mathrm{Var}(\omega,\R^N):=\left\{
\varphi\in {\rm C}^1(\overline{\omega},\R^N)\st \varphi(x)=0\mbox{ for all }x\mbox{ in } (\Gamma_D\cap{\overline{\omega}})\cup(\partial\omega\cap\Omega)
\right\}.
\end{equation*}
We note that $\varA=\mathrm{Var}(\Omega,\R^N)$.

Next, for $u_0\in\Acal$, let
\[
\Rcal(u_0) = \setb{(u_0(x),\nabla u_0(x))}{x\in\overline{\Omega}}.
\]
Then, in addition to the continuity of $F$, we assume that for some $p\in[2,\infty)$

\begin{itemize}
\item[\textbf{[H0]}] the partial derivatives of first and second order in $(y,z)$ of $F(x,y,z)$, denoted by $F_y$, $F_z$, $F_{yy}$, $F_{zz}$ and $F_{yz}$, exist and are continuous on $\overline{\Omega}\times\Ucal$ where $\Ucal$ is an open and bounded neighbourhood of $\Rcal(u_0)$ in $\R^N\times\R^{N\times d}$ for a given map $u_0\in\W^{1,p}(\Omega,\R^n)$;
\item[\textbf{[H1]}] (growth conditions) for all $x\in\overline{\Omega}$, $y\in\R^N$ and $z\in\R^{N\times d}$
\begin{align*}
&\mbox{(a)}\quad |F(x,y,z)|\leq C(y)(1+|z|^p)\\
&\mbox{(b)}\quad |F_z(x,y,z)|\leq C(y)(1+|z|^{p-1})\\
\mbox{and}\qquad& \\
&\mbox{(c)}\quad |F_y(x,y,z)|\leq C(y)(1+|z|^p),
\end{align*}
where $C(y)>0$ is in $L^\infty_{\tiny\rm loc}(\R^N)$ and depends on $F$.
\end{itemize}

\begin{remark}\label{rem:closureA}
Following \cite{GM09}, we remark that under the $p$-growth assumed in [H1] (a), the notion of strong or weak local minimisers for $I$ in $\mathcal A$ remains the same if we enlarge the space of admissible maps to
\[
\mathcal{A}^\prime:=\left\{u\in C({\overline{\Omega}},\R^N)\cap \W^{1,p}(\Omega,\R^N) \st u(x)=\bar{u}(x)\mbox{ for all }x\in\Gamma_D\right\}.
\]
This is due to the fact that $\Acal$ is dense in $\mathcal{A}^\prime$ under the topology generated by the norm \linebreak $\|u\|_{\infty}+\|\nabla u\|_{p}$ and the functional $I$ is finite and continuous on $\mathcal{A}^\prime$ under this topology.
In particular, this implies that given $\varphi\in\mathrm{Var}(\omega,\R^N)$, by extending $\varphi$ to $\Omega$ so that it takes the value $0$ in $\Omega\backslash\omega$, we may assume that $\varphi$ is in the closure of $\varA$ with respect to the above norm. Then, if $u_0$ is a strong local minimiser of $I$ in $\Acal$, the minimality condition
\[
I(u_0+ \varphi)\geq I(u_0)
\]
also holds for these maps provided that $\|\varphi\|_{\infty}$ is small enough. We will often make use of this remark without always appealing to the above argument.
\end{remark}

We use the notation $a\cdot b$ to denote the usual inner product on $\R^N$ or $\R^{N\times d}$, respectively. Note that, in both cases, we can write $a\cdot b={\rm tr}(ab^T)$. Also, we use the notation in \cite{GM09} whereby a single $x$ argument in $F$ or any of its derivatives, means that it is being evaluated at $(x,u_0(x),\nabla u_0(x))$.\vspace{0.5cm}

Under [H0] and [H1] (a) it is well known (see \cite{meyers65,BallMarsden,GM09})  that necessary conditions for a map $u_0\in\Acal$ to be a strong local minimiser of $I$ are the following:
\begin{itemize}
\item[(I)] satisfaction of the weak Euler--Lagrange equations, i.e.
\begin{equation*}\label{eq:WEL}
\int_{\Omega}[F_y(x)\cdot \varphi(x) + F_z(x)\cdot\nabla\varphi(x)]\dx = 0,
\end{equation*}
for all $\varphi\in\varA$;
\item[(II)] non-negativity of the second variation, i.e.
\begin{equation*}\label{eq:second_variation}
\int_{\Omega}[F_{yy}(x)\varphi(x)\cdot\varphi(x) + 2F_{yz}(x)\varphi(x)\cdot\nabla\varphi(x) + F_{zz}(x)\nabla\varphi(x)\cdot\nabla\varphi(x)]\dx\geq 0,
\end{equation*}
for all $\varphi\in\varA$;
\item[(III)] for all $x_0\in\overline{\Omega}$, $F(x_0,u_0(x_0),\cdot)$ is quasiconvex in the interior, i.e.
\begin{equation*}
\int_{B}\left[F(x_0,u_0(x_0),\nabla u_0(x_0) + \nabla\varphi(x)) - F(x_0)\right]\dx\geq0,
\end{equation*}
for all $\varphi\in {\rm C}^1_c(\overline{B},\R^N)$ where $B$ denotes the unit ball in $\R^d$;
\item[(IV)] for all $x_0\in\Gamma_N$ in the neighbourhood of which $\partial\Omega$ is ${\rm C}^1$, $F(x_0,u_0(x_0),\cdot)$ is quasiconvex at the (smooth) boundary, i.e.
\begin{equation*}\label{eq:qc_face}
\int_{B^-_{n(x_0)}}[F(x_0,u_0(x_0),\nabla u_0(x_0) + \nabla\varphi(x)) - F(x_0) - F_z(x_0)\cdot \nabla\varphi(x)]\dx\geq0
\end{equation*}
for all $\varphi\in V_{n(x_0)}$, where
\begin{equation*}\label{eq:Vn}
V_{n(x_0)} = \setb{\varphi\in {\rm C}^1(\overline{B^-_{n(x_0)}},\R^N)}{\varphi(x) = 0\mbox{ for all }x\in\partial B\cap\overline{B^-_{n(x_0)}}},
\end{equation*}
 $B^-_{n(x_0)}=\setb{x\in B}{x\cdot n(x_0) < 0}$, and $n(x_0)$ is the outward unit normal to $\partial\Omega$ at $x_0$.
\end{itemize}

As mentioned above, Grabovsky and Mengesha \cite{GM09} showed that, under additional hypotheses on $F$ and for $\Omega$ of class ${\rm C}^1$, a strengthened version of the above conditions is in fact sufficient for a map $u_0\in\Acal$ to be a strong local minimiser of $I$. In this work, we establish a sufficiency theorem for the case in which the domain is locally diffeomorphic to a cone and, in order to make this definition precise, we introduce some terminology.

\begin{definition}\label{def:cone}
A closed set $\mathcal{C}\subset \R^d$ is a cone with a vertex at $0\in \R^d$ if, and only if, whenever $x\in \mathcal{C}$ it also holds that
\[
tx \in \mathcal{C} \quad\mbox{for all }t > 0.
\]
Similarly, a closed set $\mathcal{C}\subset \R^d$ is a cone with a vertex at $x_0\in \R^d$ if, and only if, the set $\mathcal{C}-x_0$ is a cone with vertex at $0$. 
\end{definition}
\begin{remark}\label{remark:multiplevertices}
A cone may have more than one vertex: the set of points $\mathcal{C}:=\setb{(x,y,z)\in\R^3}{z\geq |y|}$  is a cone and any point on the $X$-coordinate axis is a vertex of $\mathcal{C}$.
\end{remark}

\begin{lemma}
Let $\mathcal{C}\subset \R^d$ be a set such that $0\in\partial \mathcal{C}$ and $\partial \mathcal{C}$ is the graph of a function $h:\R^{d-1}\to\R$ in an appropriate coordinate system. Then, $\mathcal{C}$ is a cone with vertex at $0\in \R^d$ if, and only if, $h$ is positively 1-homogeneous and, up to a change of coordinates, 
\[
\mathcal{C} = \{x=(x',x_d)\in \R^{d}\,:\, x_d \geq h(x')\},
\]
where $x'=(x_1,\ldots, x_{d-1})$.
\end{lemma}

For our purposes, we require that our domains are locally diffeomorphic to cones and we now make this definition precise.

\begin{definition}\label{def:domains}
We say that an open and bounded set $\Omega\subset\R^d$ is locally diffeomorphic to a cone if for every point $x_0\in\partial\Omega$ there exists a ball $B(x_0, r_{x_0})$ centred at $x_0$ of radius $r_{x_0}>0$ such that, up to a change of coordinate system, 
\[
\Omega\cap B(x_0,r_{x_0}) = x_0 + \{x=(x',x_d)\in B(0,r_{x_0})\,:\,g_d(x)\geq h(g'(x))\}
\]
where  $g:\overline{\Omega}\cap B(x_0,r_{x_0}) -x_0\to g(\Omega\cap B(x_0,r_{x_0}) -x_0)$ is a diffeomorphism, $g=(g_1,...,g_d)=(g',g_d)$ and $h:\R^{d-1}\to\R$ is a positively 1-homogeneous function. For simplicity, we assume $\det \nabla g(x)>0$, i.e., that $g$ is orientation-preserving. 

We also say that $\Omega\subset\R^d$ is locally a cone if for every point $x_0\in\partial\Omega$ there exist $r_{x_0}>0$ and a cone $\mathcal{C}$ with vertex at $x_0$ such that $\Omega\cap B(x_0,r_{x_0})=\mathcal{C}\cap B(x_0,r_{x_0})$ (so, the diffeomorphism $g$ that exists above is the identity in $\R^d$). 

In addition, if $\Omega\subseteq\R^d$ is locally diffeomorphic to a cone, we say that $\Omega$ has a conical boundary. 

\end{definition}

We emphasize that, for both the sufficiency and necessity results we require that our domains are locally diffeomorphic to cones. For the sufficiency result this comes from the uniform bounds required for the localization process in Theorem \ref{zhanglemma} (see \eqref{goodineq}). This appears in analogy to the localization procedure performed for the sufficiency theorem in \cite[Theorem 11.10]{GM09}, where uniform bounds between the boundary and its (flat) local blow-up are needed. For the necessity, the condition is required to build the appropriate comparison maps despite the fact that, essentially, the quasiconvexity condition arises as necessary by blowing up the minimality condition.


%
%
%

Also, as the remark below shows, the stated condition is strictly stronger than the condition that every point of the boundary admits a unique blow-up (which must then necessarily be a cone - see \cite[Proposition 2.1]{leonardi}).

\begin{remark}\label{rem:strongerbu}
We note that if a neighbourhood of $x_0\in\partial\Omega$ is diffeomorphic to a cone, i.e.
\[
\Omega\cap B(x_0,r_{x_0}) = x_0 + \{x=(x',x_d)\in B(0,r_{x_0})\,:\,g_d(x)\geq h(g'(x))\},
\]
then the blow-up cone at $x_0$ is precisely the cone $\mathcal{C}$ defined by
\[
\mathcal{C} = \{x=(x',x_d)\in \R^{d}\,:\, x_d \geq h(x')\}.
\]
Indeed, if the blow-up of $\Omega\cap B(x_0,r_{x_0})$ at $x_0$ exists, then it is the same as the blow-up set of $\Omega$ at $x_0$. A direct computation shows that
\begin{align*}
\frac{\Omega\cap B(x_0,r_{x_0}) - x_0}{\eps} & = \left\{\left(\frac{x'}{\eps},\frac{x_d}{\eps}\right)\,:\,g_d(x)\geq h(g'(x))\right\}\\
& = \left\{(y',y_d)\,:\,g_d(\eps y)\geq h(g'(\eps y))\right\}.
\end{align*}
Up to affine transformation, we may assume that $g(0)=0$ and $\nabla g(0) = {\rm I}_d$. Denoting by $\nabla_y$ the one-sided derivative in the direction $y$, we infer that
\begin{align*}
\frac{\Omega\cap B(x_0,r_{x_0}) - x_0}{\eps} & = \left\{(y',y_d)\,:\,\frac{g_d(\eps y) - g_d(0)}{\eps}\geq \frac{h(g'(\eps y)) - h(g'(0))}{\eps}\right\}\\
& \overset{\eps\to0^+}{\longrightarrow} \left\{(y',y_d)\,:\,\nabla g_d(0)\cdot y\geq \nabla_y [h\circ g'](0)\right\}\\
& = \left\{(y',y_d)\,:\, y_d\geq h(y')\right\},
\end{align*}
where the last equality follows from the 1-homogeneity of $h$ as
\begin{align*}
 \nabla_y [h\circ g'](0) & = \lim_{\eps\to0^+} \frac{h(g'(\eps y)) -  h(g'(0))}{\eps}  =  \lim_{\eps\to0^+} h\left(\frac{g'(\eps y) - g'(0)}{\eps}\right) = h(\nabla g'(0)y) = h(y').
\end{align*}
Hence, $\Omega$ being locally diffeomorphic to a cone is stronger than the condition that it admits unique blow-ups at every boundary point. In fact, it is strictly stronger as the following example shows.

Consider the set
\begin{equation*}
\Omega:=\left\{ (x,y)\in\R^2\,:\,x\in[-1,1]\backslash\{0\}\mbox{ and }-\sqrt{1-x^2}\leq y\leq x^2\sin\left(\frac{1}{x}\right)\right\}\bigcup\{(0,0)\}
\end{equation*}
Then, $(0,0)\in\partial\Omega$ and $\partial\Omega$ is Lipschitz. Also, $\Omega$ blows-up around $(0,0)$ into the lower half-space in $\R^2$, however, $\Omega$ is not diffeomorphic to a cone in any neighbourhood of $(0,0)$ as, locally, $\partial\Omega$ is the graph of the function $f(x)=x^2\sin\left(\frac{1}{x}\right)$, which is not ${\rm C}^1$.
\end{remark}




We may now define the quasiconvexity at the boundary conditions and recall the classical quasiconvexity in the interior:

\begin{definition}\label{def:qc}
Let $F\colon\R^{N\times d}\to\R$, $A\in\R^{N\times d}$ and denote by $B$ the unit ball in $\R^d$.
\begin{itemize}
\item We say that $F$ is quasiconvex in the interior at $A$ if
\begin{equation*} 
\int_{B}\left[F(A + \nabla\varphi(x)) - F(A)\right]\dx\geq0
\end{equation*}
for all $\varphi\in {\rm C}^1_c({B},\R^N)$.
\item Let $\mathcal{C}$ be a cone with vertex at $0$. We say that $F$ is quasiconvex at $A$ at the boundary 
with associated cone $\mathcal{C}$ if
\begin{equation*}\label{eq:qc_bd1}
\int_{B_{\mathcal{C}}}\left[F(A + \nabla\varphi(x)) - F(A)-F'(A)\cdot\nabla\varphi\right]\dx\geq0
\end{equation*}
for all $\varphi\in V_{\mathcal{C}}$. Here, 
$B_{\mathcal{C}}:=B\cap \mathcal{C}$ and 
 \begin{equation*}
V_{\mathcal{C}} :=\setb{\varphi\in {\rm C}^1(\overline{B_{\mathcal{C}}},\R^N)}{\varphi(x) = 0\mbox{ for all }x\in \partial B\cap\overline{\mathcal{C}}}.
\end{equation*}
\end{itemize}
\end{definition}



\begin{figure}[h]
\centering
\includegraphics[width=0.6\textwidth]{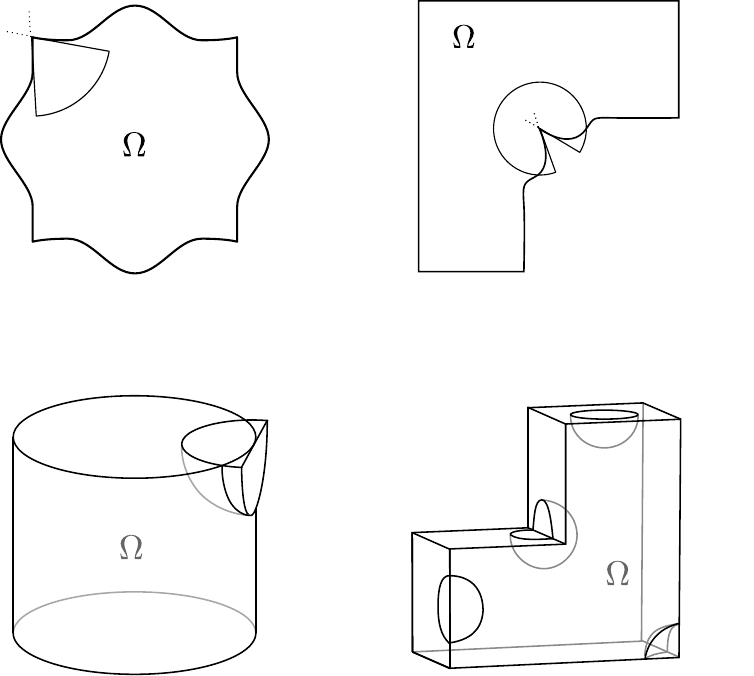}
\caption{A visualisation of the sets $B_{\mathcal{C}}$ in the definition of quasiconvexity at the boundary, translated and scaled onto the respective parts of $\Omega$. The corresponding space of test functions $V_{\mathcal{C}}$ simply amounts to smooth maps on the respective section of the unit ball which vanish on the part of the boundary that also belongs to the closed unit ball; that is, they vanish on the curved boundary whereas the conical part of the boundary is free.}
\label{fig:example1}
\end{figure}


\begin{remark}\label{remark:p-quasiconvexity}
Note that, by [H1], $F$ satisfies a $p$-growth and hence the quasiconvexity conditions in the interior and at the boundary in Definition \ref{def:qc} also hold for test functions $\varphi\in \overline{{\rm C}^1_{c}(B,\R^N)^{\W^{1,p}}}=\W^{1,p}_0(B,\R^N)$ and $\varphi\in \overline{V_{\mathcal{C}}\,^{\W^{1,p}}}$, respectively, where the closure is taken in the strong topology of $\W^{1,p}$. The proof in the case of the interior can be found in \cite[Proposition 2.4]{ballmurat} and the proof for the conditions at the boundary is essentially the same. This remark is crucial in the analysis that follows. As with Remark \ref{rem:closureA}, if $\varphi\in {\rm C}^1_c(\omega,\R^N)$ for some $\omega\subset B$, by extending $\varphi$ to $B$ such that it takes the value 0 outside $B\backslash\omega$, we obtain a map in $\W^{1,p}_0(B,\R^N)$ and the quasiconvexity condition still holds for these maps. The analogous situation holds for the case of quasiconvexity at the boundary and the space $\overline{V_{\mathcal{C}}\,^{\W^{1,p}}}$.
\end{remark}

It is also well-known (see e.g. \cite[Remark 5.1]{Giusti}) that quasiconvexity in the interior is independent of the particular choice of the set on which integration is 
performed (in our case the unit ball $B$). Indeed, if the condition holds for $B\subset\mathbb{R}^d$, it also holds for any bounded, open set $D\subset\mathbb{R}^d$ with $\mathcal{L}^d(\partial D) = 0$.
Similarly, quasiconvexity at the boundary also does not depend on the specific form of the sets $B_{\mathcal{C}}$ and we make this precise in the following definition and lemma, motivated by \cite{BallMarsden}. 

\begin{definition}\label{def:standard}
A standard conical-boundary region with associated cone $\mathcal{C}$ is a bounded Lipschitz domain $D\subset\R^d$ such that there exists $d_0\in \partial D$ satisfying:
\begin{itemize}
\item[(i)] $D-d_0$ is contained in $\mathcal{C}\subseteq\R^d$ with vertex $0$; 
\item[(ii)] for each $x\in \mathcal{C}$, the $1$-dimensional interior of the set
\[
 D\cap \setb{d_0+tx}{t>0}
\]
is non-empty.
\end{itemize}  
\end{definition}

\begin{lemma}\label{lemma:qcindependent}
Let $F\colon\R^{N\times d}\to\R$ and $A\in\R^{N\times d}$. If $F$ is quasiconvex at $A$ at the boundary with associated cone $\mathcal{C}$, then, for any standard conical-boundary region $D$ with associated cone $\mathcal{C}$, it holds that
\[
\int_{D}[F(A + \nabla\varphi(x)) - F(A) - F'(A)\cdot \nabla\varphi(x)]\dx\geq0,
\]
for all $\varphi\in V_D$, where
\[
V_D = \setb{\varphi\in {\rm C}^1(\overline{D},\R^N)}{\varphi(x) = 0\mbox{ for all }x\in\partial D\cap \left(d_0+ \mathrm{int}\,\mathcal{C}\right)}.
\]
\end{lemma}

\begin{proof}
Suppose that $F$ is quasiconvex at $A$ at the boundary with associated cone $\mathcal{C}$ and let $D$ be a corresponding standard conical-boundary region. By Definition \ref{def:standard}, there exists $d_0\in\partial D$ such that $D-d_0\subset \mathcal{C}$. Then, for $\varepsilon>0$ small enough, we can ensure that
\[
\varepsilon(D-d_0) \subset B\cap \mathcal{C}.
\]
Note that the \textit{conical} part of the boundary of $\varepsilon(D-d_0) $ is still contained in $\partial\mathcal{C}$.

Therefore, if $\varphi\in V_D$ and we define $\psi: B\cap \mathcal{C}\to \mathbb{R}^N$ by
\[
\psi(x) = \left\{\begin{array}{cc}
\varepsilon \varphi(d_0+ x/\varepsilon), & x\in \varepsilon \left(D - d_0\right)\\
0, & \mbox{otherwise},
\end{array}\right.
\]
then $\psi \in \overline{V_{\mathcal{C}}\,^{\W^{1,p}}}$. Hence, by Remark \ref{remark:p-quasiconvexity}, the quasiconvexity condition applies for $\psi$ to give
\begin{align*}
0 & \leq \int_{B\cap \mathcal{C}}[F(A+\nabla\psi(x)) - F(A) - F'(A)\cdot\nabla\psi(x)]\dx \\
 & = \int_{\varepsilon(D-d_0)} [F(A+\nabla\varphi(d_0+x/\eps))-F(A)-F'(A)\cdot\nabla\varphi(d_0+x/\eps)]\dx \\ 
&= \eps^d\int_D [F(A+\nabla\varphi(y))-F(A) -F'(A)\cdot\nabla\varphi(y)]\dy.
\end{align*}

\end{proof}


\section{Necessary conditions}\label{sec:necessity}
Equipped with all the required definitions, we proceed to discuss the necessity of quasiconvexity at the boundary for a map $u_0\in\Acal$ to be a strong local minimiser of the functional $I$. We recall that
\[
I(u) = \int_{\Omega}F(x,u(x),\nabla u(x))\dx
\]
and
\[
\Acal = \setb{u\in {\rm C}^1(\overline{\Omega},\R^N)}{u(x) = \bar{u}(x)\mbox{ for all } x\in \Gamma_D}.
\]
We begin by compiling all the necessary conditions in one theorem. We also recall that a single $x$ argument in $F$, or any of its derivatives, corresponds to the triple $(x,u_0(x),\nabla u_0(x))$.

\begin{theorem}\label{thm:necessity}
Let $\Omega\subset\R^d$ be locally diffeomorphic to a cone. For $u_0\in\Acal$, assume that $F\colon\overline{\Omega}\times\R^N\times\R^{N\times d}\to\R$ is continuous and satisfies [H0] and [H1]. 
If $u_0$ is a strong local minimiser of $I$ in $\Acal$, then the following hold:
\begin{itemize}
\item[$(I)$] $u_0$ satisfies the weak Euler--Lagrange equations, i.e.
\begin{equation*}
\int_{\Omega}[F_y(x)\cdot \varphi(x) + F_z(x)\cdot\nabla\varphi(x)]\dx = 0,
\end{equation*}
for all $\varphi\in\varA$;
\item[$(II)$] the second variation at $u_0$ is non-negative, i.e.~
\begin{equation*}
\int_{\Omega}[F_{yy}(x)\varphi(x)\cdot\varphi(x) + 2F_{yz}(x)\nabla\varphi(x)\cdot\varphi(x) + F_{zz}(x)\nabla\varphi(x)\cdot\nabla\varphi(x)]\dx\geq 0,
\end{equation*}
for all $\varphi\in\varA$;
\item[$(III)$] for all $x_0\in\overline{\Omega}$, $F(x_0,u_0(x_0),\cdot)$ is quasiconvex in the interior, i.e.
\begin{equation*}
\int_{B}[F(x_0,u_0(x_0),\nabla u_0(x_0) + \nabla\varphi(x)) - F(x_0)]\dx \geq 0,
\end{equation*}
for all $\varphi\in {\rm C}^1_c(B,\R^N)$;
\item[$(IV)$] for all $x_0\in\Gamma_N$, a neighbourhood of which is diffeomorphic to a cone $\mathcal{C}$, $F(x_0,u_0(x_0),\cdot)$ is quasiconvex at the boundary with associated cone $\mathcal{C}$, i.e.
\begin{equation*}
\int_{B_{\mathcal{C}}}[F(x_0,u_0(x_0),\nabla u_0(x_0) + \nabla\varphi(x)) - F(x_0) - F_z(x_0)\cdot \nabla \varphi(x)]\dx\geq 0,
\end{equation*}
for all $\varphi\in V_{\mathcal{C}}$.
\end{itemize}
\end{theorem}

\begin{proof}[Proof of Theorem \ref{thm:necessity}] The proofs of $(I)$, $(II)$ are standard and can be found in e.g. \cite{dacorognabook}. 
The necessity of $(III)$ is essentially due to Meyers \cite{meyers65} (see also \cite{BallMarsden}). The proof of $(IV)$ with $\mathcal{C}$ a half space, amounting to the standard quasiconvexity at the boundary condition, can be found in Ball \& Marsden~\cite{BallMarsden}. Nevertheless, here, we give a variant of the proof which is also appropriate for proving the quasiconvexity at the singular points of the boundary.
To this end, let $x_0\in\Gamma_N$ and let $g$ be the respective diffeomorphism between $\Omega\cap B(x_0,r_{x_0})$ and a cone $\mathcal{C}\cap B$. Without loss of generality we assume that $\nabla g(x_0) = {\rm I}_d$. Also, let $\varphi\in V_{\mathcal{C}}$ and define
\[
u_\varepsilon(x) = \left\{\begin{array}{ll}
u_0(x) + \varepsilon\varphi\left(\frac{g(x) - g(x_0)}{\varepsilon}\right), & x\in g^{-1}(g(x_0)+\varepsilon B_{\mathcal{C}})\\
u_0(x), & \mbox{otherwise in $\Omega$.}
\end{array}\right.
\]
We remark that for $\varepsilon>0$ small enough
\begin{align*}
&g^{-1}(g(x_0)+\varepsilon B_{\mathcal{C}})\subset\Omega, \\
&g^{-1}(g(x_0)+\varepsilon (\partial B\cap\overline{B_{\mathcal{C}}}))\subset\Omega \mbox{ and} \\
&g^{-1}(g(x_0)+\varepsilon ( B\cap\partial B_\mathcal{C}))\subset \Gamma_N.
\end{align*}

Hence, the function $\psi_\eps = u_\eps - u_0$ lies in the closure of ${\rm Var}(\Acal)$ in the topology generated by the norm $\|u\|_\infty + \|\nabla u\|_p$ and 
$||\psi_\eps||_{\infty} \to 0$, as $\varepsilon\to 0$. By Remark \ref{rem:closureA} and the fact that $u_0$ is a strong local minimiser, we infer that
\begin{align*}
0 &\leq I(u_\varepsilon) - I(u_0).
\end{align*}
Also, since $u_0$ is a strong local minimiser, it satisfies the weak Euler-Lagrange equations in $(I)$ with $\psi_\eps$ as a test function, i.e.
\begin{equation*}
\int_{\Omega}[F_y(x)\cdot \psi_\eps(x) + F_z(x)\cdot\nabla\psi_\eps(x)]\dx = 0.
\end{equation*}
Consequently, the minimality condition results in
\[
0\leq \int_{g^{-1}(g(x_0)+\varepsilon B_{\mathcal{C}})}\left[F(x,u_\varepsilon(x),\nabla u_\varepsilon(x))  -F(x) - F_y(x)\cdot \psi_\eps(x) - F_z(x)\cdot\nabla\psi_\eps(x) \right]\dx.
\]
Changing variables to 
\[y_\varepsilon(x)= (g(x)-g(x_0))/\varepsilon
\]
with inverse $x_\varepsilon(y)=g^{-1}(g(x_0)+\varepsilon y)$ and $\nabla x_\varepsilon(y) = \varepsilon\nabla g^{-1}(g(x_0)+\varepsilon y)$ we deduce that
\begin{align*}
0& \leq \varepsilon^d\int_{B_{\mathcal{C}}}\left[F(x_\varepsilon(y), u_0(x_\varepsilon(y))+\varepsilon\varphi(y),\nabla u_0(x_\varepsilon(y))+\nabla\varphi(y)\nabla g(x_\varepsilon(y))) \right.\\
&\qquad\left. - F(x_\varepsilon(y)) - \eps F_y(x_\varepsilon(y))\cdot \varphi(y) - F_z(x_\varepsilon(y))\cdot\nabla\varphi(y)\nabla g(x_\varepsilon(y))\right]\det[\nabla g^{-1}(g(x_0)+\varepsilon  y)]\dy.
\end{align*}
Dividing by $\varepsilon^d$ and sending $\varepsilon\to 0$, bounded convergence implies that
\begin{align*}
0 & \leq
 \int_{B_{\mathcal{C}}}\left[F(x_0,u_0(x_0),\nabla u_0(x_0)+\nabla\varphi(y)) - F(x_0)\right]\det[\nabla g^{-1}(g(x_0))]\dy\\
& = \int_{B_{\mathcal{C}}}\left[F(x_0,u_0(x_0),\nabla u_0(x_0)+\nabla\varphi(y)) - F(x_0) - F_z(x_0)\cdot\nabla\varphi(y)\right]\dy,
\end{align*}
noting that $\nabla g^{-1}(g(x_0))=(\nabla g(x_0))^{-1} = {\rm I}_d$.
\end{proof}
We remark here that the assumption on the domains being locally diffeomorphic to a cone is used to construct appropriate comparison maps to obtain the quasiconvexity at the boundary as a necessary condition. However, having local $\W^{1,1}$-homeomorphisms between $\Omega$ and a cone would be enough to obtain the conclusion (see \cite{AmbDMchainr}). 


\section{Sufficient conditions}\label{sec:sufficiency}

In this section we prove our main result in the form of a quasiconvexity-based sufficiency theorem for a map $u_0\in\Acal$ to be a strong local minimiser of $I$ in $\Acal$. For this we require additional coercivity conditions on the integrand $F$ which we state here. We recall that $p\in[2,\infty)$.

\begin{itemize}
 \item[\textbf{[H2]}] (coercivity conditions) $F$ is bounded from below. If $p=2$, we assume that
\begin{equation*}
\int_{\Omega} F(x,u(x),\nabla u(x))\dx\geq c_2(r)||u||^2_{1,2} - c_1(r)
\end{equation*}
for all $u\in\Acal$ such that $||u||_{\infty}\leq r$, where $c_1(r)>0$ and $c_2(r)>0$ are locally bounded. If $p>2$, we assume that for all $\varphi\in\varA$ with $||\varphi||_\infty \leq r$,
\begin{equation*}
\int_{\Omega}[F(x,u_0(x)+\varphi(x),\nabla u_0(x)+\nabla\varphi(x)) - F(x)]\dx\geq c_1(r)||\nabla\varphi ||^p_{p} - c_2(r)||\nabla\varphi ||^2_{2}
\end{equation*}
for some $c_1(r)>0$, $c_2(r)>0$ which are locally bounded. Note that the latter condition need only apply to $u_0$.
\end{itemize}

Lastly, as in \cite{GM09}, we require an additional uniform continuity assumption:

\begin{itemize}
\item[\textbf{[UC]}] For every $\eps>0$ there exists $\delta>0$ such that, for every $z\in\R^{N\times d}$ and $x,x_0\in\overline{\Omega}$ with $|x-x_0|<\delta$, it holds that
\begin{equation*}
|F(x_0,u_0(x_0),\nabla u_0(x_0)+z)-F(x,u_0(x),\nabla u_0(x)+z)|<\varepsilon(1+|z|^p).
\end{equation*}
\end{itemize}

For brevity we also introduce the function $S$ below.

\begin{definition}
For $k\in\N$, let $S\colon\R^k\to\R$ denote the function
\[
S(\xi) = \left(|\xi|^2 + |\xi|^p\right)^{\frac{1}{2}}.
\] 
\end{definition}

We are now in a position to state our main theorem:

\begin{theorem}[Sufficiency Theorem]\label{thm:sufficiency}
Let $\Omega\subset\R^d$ be locally diffeomorphic to a cone. Let $u_0\in\Acal$ and assume that $F\colon\overline{\Omega}\times\R^N\times\R^{N\times d}\to\R$ is continuous and satisfies [H0]--[H2]. Further assume that for some $c_0>0$ the following hold:
\begin{itemize}
\item[$(I)$] $u_0$ satisfies the weak Euler--Lagrange equations, i.e., for all $\varphi\in\varA$,
\begin{equation*}
\int_{\Omega}[F_y(x)\cdot \varphi(x) + F_z(x)\cdot\nabla\varphi(x)]\dx = 0;
\end{equation*}
\item[$(II)$] the second variation at $u_0$ is strongly positive, i.e.,~ for all $\varphi\in\varA$,
\begin{equation*}
\int_{\Omega}[F_{yy}(x)\varphi(x)\cdot\varphi(x) + 2F_{yz}(x)\nabla\varphi(x)\cdot\varphi(x) + F_{zz}(x)\nabla\varphi(x)\cdot\nabla\varphi(x)]\dx\geq c_0 ||\nabla\varphi ||^2_2;
\end{equation*}
\item[$(III)$] for all $x_0\in\overline{\Omega}$, $F(x_0,u_0(x_0),\cdot)$ is strongly quasiconvex in the interior, i.e.,  for all \linebreak $\varphi\in~{\rm C}^1_c({B},\R^N)$,
\begin{equation*}
\int_{B}\left[F(x_0,u_0(x_0),\nabla u_0(x_0) + \nabla\varphi(x)) - F(x_0)\right] \dx\geq c_0 ||S(\nabla \varphi)||_2^2;
\end{equation*}
\item[$(IV)$] for all free-boundary points $x_0\in\Gamma_N$, $F(x_0,u_0(x_0),\cdot)$ is strongly quasiconvex at the boundary, i.e. for all $\varphi\in V_{\mathcal{C}}$, with $\mathcal{C}$ the cone associated to $x_0$ given by Definition \ref{def:domains},
\begin{equation*}
\int_{B_{\mathcal{C}}}\left[F(x_0,u_0(x_0),\nabla u_0(x_0) + \nabla\varphi(x)) - F(x_0)-F_z(x_0)\cdot \nabla\varphi \right]\dx\geq c_0 ||S(\nabla \varphi)||_2^2;
\end{equation*}
\item[$(V)$] $u_0$ satisfies [UC].
\end{itemize}
Then, $u_0$ is a strong local minimiser of $I$ in $\Acal$.
\end{theorem}

The proof of the Sufficiency Theorem is based on the G{\aa}rding inequality given by Theorem \ref{zhanglemma} below. In its original form due to Zhang \cite{zhang92QCremarks}, the result asserts that, under the quasiconvexity assumptions, any ${\rm C}^1$ solution of the Euler--Lagrange equations is a spatially-local minimiser. In the presence of lower order terms in $F$, we obtain a variant - equivalent to a G{\aa}rding inequality - appropriate for our purposes.
 
\begin{theorem}\label{zhanglemma}
Let $\Omega\subset\R^d$ be locally diffeomorphic to a cone. Let $u_0\in~{\rm C}^1(\overline{\Omega},\R^N)$ and assume that $F\colon\overline{\Omega}\times\R^N\times\R^{N\times d}\to\R$ is continuous and satisfies [H0]--[H2]. 
In the notation of Theorem \ref{thm:sufficiency}, suppose that conditions $(I)$, $(III)$ and $(IV)$ are satisfied. Then, there exist some $R>0$ and $\delta>0$ such that, denoting by $\Omega(x_0,R):=\Omega\cap B(x_0,R)$, the following hold:
\begin{itemize}
\item[(1)] for all $x_0\in\overline{\Omega}$
\begin{align}\label{eq:minimality}
&\underset{\Omega(x_0,R)}{\int}\left(\frac{c_0}{2}|S(\nabla\varphi(x))|^2 - C\, |\varphi(x)|^2\right)\dx\notag\\
\leq  & \underset{\Omega(x_0,R)}{\int}\left(F(x,u_0(x)+\varphi(x),\nabla u_0(x)+\nabla\varphi(x))-F(x) \right)\dx \tag{$\ast$}
\end{align}
for all $\varphi\in {\rm W}^{1,p}_0(\Omega(x_0,R),\mathbb{R}^N)$ with $\|\varphi\|_{\LL^\infty}<\delta$. Here, $c_0>0$ is the constant that appears in conditions (III) and (IV) and $C>0$ is a constant with $C=C(d,\Omega,\|u\|_\infty)$;
\item[(2)] for all $x_0\in\overline{\Omega}$ such that $\Omega(x_0,R)\cap\Gamma_N\neq \emptyset$ and $\Omega(x_0,R)$ is diffeomorphic to a cone $\mathcal{C}$, \eqref{eq:minimality} holds for every function $\varphi\in \overline{{\rm Var}(\Omega(x_0,R),\mathbb{R}^N)^{{\rm W}^{1,p}}}$ satisfying  $\|\varphi\|_{\LL^\infty}<\delta$, where the closure is taken in the strong topology of ${\rm W}^{1,p}(\Omega,\R^N)$.
\end{itemize}
\end{theorem}
Before proceeding with the proof of this theorem, we state the following technical lemma regarding the growth conditions that the shift of the integrand $F$ satisfies. We postpone the proof of the lemma until the Appendix in Section \ref{Appendix}.

\begin{lemma} \label{LemmagrowthG}
Let $\Omega\subset\R^d$ be a Lipschitz domain. Assume further that $F\colon\overline{\Omega}\times \R^N\times \R^{N\times d}\rightarrow\R$ is a continuous integrand, $u_0\in {\rm C}^1(\overline{\Omega},\R^N)$ and that [H0], [H1] and [UC] hold. 
Define the function $G\colon\overline{\Omega}\times \R^N\times \R^{N\times d}\rightarrow\R$  by 
\begin{align*}
G(x,y,z):=&F(x,u_0(x)+y,\nabla u_0(x)+z)-F(x) -F_y(x)\cdot y-F_z(x)\cdot z\notag\\
=&\int_0^1 (1-t)L(x, u_0(x)+ty,\nabla u_0(x)+tz)[(y,z),(y,z)],
\end{align*}
where the bilinear form $L(x,v,w)$ is given by 
\begin{align*}
L(x,v,w)[(y,z),(\hat{y},\hat{z})]:= &F_{yy}(x,v,w)y\cdot \hat y+ F_{yz}(x,v,w)y\cdot z+F_{yz}(x,v,w)\hat{y}\cdot \hat{z} \\&+F_{zz}(x,v,w)z\cdot \hat{z}.
\end{align*}

The following hold:

\begin{itemize}
\item[(a)] for each $x\in\overline{\Omega}$, $y,\hat y\in\R^N$ and $z,\hat z\in\R^{N\times d}$ and for some locally bounded function $C(y,\hat y)$ on $\R^{2N}$, 
\[
|G(x,y,z) - G(x,\hat y,\hat z)|\leq C(y,\hat y)\left(A_{p-1}(y,z,\hat y,\hat z)|z-\hat z| + A_p(y,z,\hat y,\hat z)|y-\hat y|\right),
\]
where
\[
A_p(y,z,\hat y,\hat z) = |y|+|\hat y|+|z|+|\hat z|+|z|^p+|\hat z|^p.
\]
In particular, for some locally bounded function $C(y)$,
\[
|G(x,y,z)|\leq  C(y)\left(|S(y)|^2+|S(z)|^2\right);
\] 
\item[(b)]  for every $\varepsilon>0$ there exist $R=R(\varepsilon)>0$ and $\delta=\delta(\varepsilon)>0$ such that, for all $x_0,x\in\overline{\Omega}$, $z\in\R^{N\times d}$ and $J\in\R^{d\times d}$, if $|x-x_0|<R$,  $w=z\, J$ and $|y|+|J-\mathrm{I}_d|<\delta$, then
\begin{equation}\label{quaderror}
|G(x_0,0,w)-G(x,y,z)|\det J||<c|y|^2|\det J|+\varepsilon |S(z)|^2|\det J|
\end{equation}
for some constant $c>0$ where $\mathrm{I}_d$ denotes the $d\times d$ identity matrix;
\item[(c)] for every $\varepsilon>0$ there exists $\delta=\delta(\varepsilon)>0$ such that, for  all $z\in\R^{N\times d}$ and $J\in\R^{d\times d}$, if  $w=z\, J$ and $|J-\mathrm{I}_d|<\delta$, then 
\begin{align}
c_0\left||S(w)|^2-|S(z)|^2|\det J|\right|< & \,{\varepsilon} |S(z)|^2.\label{auxineq}
\end{align}
 \end{itemize}
\end{lemma}

We next proceed to prove Theorem \ref{zhanglemma} for the case in which $\Omega$ is locally a cone, with the purpose of presenting more clearly the ideas behind this result. The reader is referred to Section \ref{sec:AppendDiffeom} for the more technical part involving domains locally diffeomorphic to a cone:

\begin{proof}[Proof of Theorem \ref{zhanglemma} (when $\Omega$ is locally a cone)]
Take $J=\mathrm{I}_d$ and $\varepsilon=\frac{c_0}{2}$ in Lemma \ref{LemmagrowthG}, with $c_0>0$ as in assumptions $(III)$ and $(IV)$. Then, there exist $R>0$ and $\delta>0$ such that for any $x_0 \in\overline{\Omega}$, $\varphi\in\overline{\mathrm{Var}(\Omega(x_0,R),\R^N)^{\W^{1,p}}}$ with $\|\varphi\|_{\LL^\infty}<\delta$, we can ensure that
\begin{equation}\label{eq:zhangproof0}
G(x_0,0,\nabla\varphi(x))-G(x,\varphi(x),\nabla\varphi(x))\leq C|\varphi(x)|^2+\frac{c_0}{2}|S(\nabla\varphi(x))|^2
\end{equation}
for every $x\in \Omega(x_0,R)$. We remark here that $J$ plays the role of the Jacobian of the diffeomorphism locally mapping $\Omega$ into a cone. Hence, if $\Omega$ is locally a cone $J=\mathrm{I}_d$.

It is now straightforward to conclude the proof. Note that for $x_0\in\Omega$ or $x_0\in\Gamma_D$ with $\Omega(x_0,R)\cap\Gamma_N=\emptyset$, the quasiconvexity condition $(III)$ remains valid if $B$ is replaced by $\Omega(x_0,R)$. Similarly, if $\Omega(x_0,R)\cap\Gamma_N\neq\emptyset$ (in particular, if $x_0\in\Gamma_N$) then, by assumption,
\[
\Omega(x_0,R) = \mathcal{C}\cap B(x_0,R)
\]
for some cone $\mathcal{C}$ and Lemma \ref{lemma:qcindependent} now states that the quasiconvexity condition $(IV)$ remains valid for $\Omega(x_0,R)$. Of course, this is also true for the strong quasiconvexity condition at hand. Then, $\varphi\in\overline{{\rm Var}(\Omega(x_0,R),\R^N)^{\W^{1,p}}}$ is an appropriate test function by Remark \ref{remark:p-quasiconvexity} and we obtain that

\begin{align*}
c_0\underset{\Omega(x_0,R)}{\int}|S(\nabla\varphi(x))|^2\dx\leq & \underset{\Omega(x_0,R)}{\int}G(x_0,0,\nabla\varphi(x))\dx\notag\\
\leq& \underset{\Omega(x_0,R)}{\int}\left(G(x,\varphi(x),\nabla\varphi(x))+\frac{c_0}{2}|S(\nabla\varphi(x))|^2 \right)\dx \\ 
&+C \underset{\Omega(x_0,R)}{\int}|\varphi(x)|^2\dx.
\end{align*}
Since $u_0$ satisfies the Euler-Lagrange equations, we deduce that 
\begin{align*}
&\frac{c_0}{2}\underset{\Omega(x_0,R)}{\int}\left(|S(\nabla\varphi(x))|^2-C|\varphi(x)|^2\right)\dx\\
\leq &  \underset{\Omega(x_0,R)}{\int}\left(F(x,u_0(x)+\varphi(x),\nabla u_0(x)+\nabla\varphi(x))-F(x) \right)\dx,
\end{align*}
which is the required inequality.
\end{proof}

\begin{remark}\label{remark:cubes}
Let $\Omega_Q(x_0, R) := \Omega\cap Q(x_0, R)$, where $Q(x_0, R)$ is a cube centred at $x_0$, with sides parallel to the coordinate axes and side length $2R$. Then, it is easy to see that, for a function $\varphi\in \overline{{\rm Var}(\Omega_Q (x_0, R),\R^N)^{\W^{1,p}}}$, we can assign the value of $0$ in $\Omega(x_0, 2R)\setminus\Omega_Q(x_0, R)$ and hence assume that $\varphi\in \overline{{\rm Var}(\Omega(x_0, 2R),\R^N)^{\W^{1,p}}}$. Therefore, Theorem \ref{zhanglemma} still remains valid if we exchange $\Omega(x_0,R)$ by $\Omega_Q(x_0,R)$ in the statement.
\end{remark} 

The following result provides a global estimate that derives from the G{\aa}rding inequality in Theorem \ref{zhanglemma}. 

\begin{proposition}\label{global_garding}
Let $(h_k)\subseteq \overline{{\rm Var}(\Omega,\R^N)^{\W^{1,p}}}\cap \LL^{\infty}(\Omega,\R^N)$ and $(a_k)\subseteq\R$ be  sequences such that 
\begin{itemize}
\item $\| h_k\|_{\LL^\infty}\overset{k\rightarrow\infty}{\longrightarrow}0$;
\item $a_k^{-1} S( h_k) \overset{k\rightarrow\infty}{\longrightarrow}0$  in $\LL^2(\Omega)$;
\item $a_k^{-1} S(\nabla h_k)$ is bounded in $\LL^2(\Omega)$. 
\end{itemize} 
Then, 
\begin{align*}
\liminf_{k\rightarrow\infty}\frac{c_0}{2}a_k^{-2}\underset{\Omega}{\int}|S(\nabla h_k(x))|^2\dx 
\leq &\liminf_{k\rightarrow\infty}\,a_k^{-2}\underset{\Omega}{\int}G(x,h_k(x),\nabla h_k(x))\dx.
\end{align*}
\end{proposition} 
\begin{proof}
\noindent\emph{\bf Step 1.} The first step will be to establish that we can cover $\Omega$ by a finite sequence of rectangles $(\mathcal{R}_j)_{j\in J}$, on which we can apply the G{\aa}rding inequality from Theorem \ref{zhanglemma} and such that their interiors are pairwise disjoint. Furthermore, we construct $(\mathcal{R}_j)$ in such a way that, for a given Radon measure $\mu$ on $\R^d$, 
\begin{equation}\label{eq:mu_sing0}
\mu(\partial\mathcal{R}_j)=0.
\end{equation}

Let $R>0$ as in Theorem \ref{zhanglemma}. Since $\Omega$ is locally diffeomorphic to a cone, for each $x\in\overline{\Omega}$ we may find a cube $Q(x,r_x)$ of side-length $r_x<\frac{R}{2}$, with sides parallel to the coordinate axes, such that the diffeomorphism from Definition \ref{def:domains} exists for $Q(x,2r_x)$ (note the doubling of the side-length of the cubes). This is a cover for $\overline{\Omega}$ and, by compactness, we may extract a finite subcover with the same property. 

We may split this finite cover into rectangles with pairwise disjoint interiors. Indeed, to achieve this, we can extend the faces of the cubes in the finite cover and then take the rectangles that result enclosed within these extensions, as in Figure \ref{fig:cover}. Note that, if we dilate  by a factor of two the rectangles  intersecting $\partial\Omega$, the dilation being with respect to their center, in each of these dilations there will also be a diffeomorphism taking that portion of $\Omega$ into a cone. This follows from the fact that such a diffeomorphism exists for all the cubes of the form $Q(x,2r_x)$ such that they intersect $\partial\Omega$.  

\begin{figure}[h]
\centering
\includegraphics[width=1.0\textwidth]{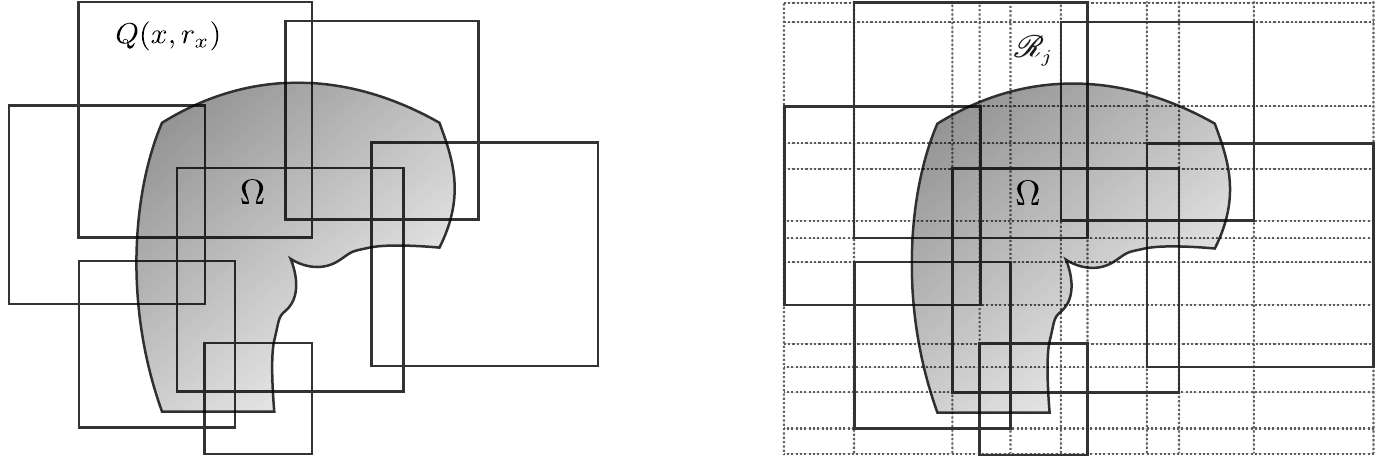}
\caption{On the left: Finite cover of $\Omega$ with cubes $Q(x,r_x)$. On the right: Finite cover of $\Omega$ with the rectangles having disjoint interiors, generated by the cubes. Note that, for simplicity, in this figure the finite cover with cubes $Q(x,r_x)$ is not represented so that it satisfies the condition of $\Omega\cap Q(x,2r_x)$ being diffeomorphic to a cone. However, the construction of the rectangles $\mathcal{R}_j$ can be done in the exact same way if we start with a finer cover of cubes. }
\label{fig:cover}
\end{figure}

Whereby, we have $J\subseteq \N$ and pairwise disjoint open rectangles $(\mathcal{R}_j)_{j\in J}$ such that
\begin{equation*}
\Omega\subseteq\bigcup_{j\in J} \overline{\mathcal{R}_j}.
\end{equation*}
Condition (\ref{eq:mu_sing0}) can be achieved by noting that at most a countable number of hyperplanes parallel to the coordinate axes can admit positive $\mu$-measure. Hence, the initial cover $Q(x,r_x)$ can be taken by requiring that the faces of each $Q(x,r_x)$ are contained in hyperplanes so that, when intersected with a neighbourhood of $\Omega$, they have null $\mu$-measure. Note that the faces of the rectangles $\mathcal{R}_j$ are necessarily contained in the same hyperplanes as those that contain the faces of the cubes $Q(x,r_x)$. Since these hyperplanes have null $\mu$-measure, the faces of each $\mathcal{R}_j$ will also have null $\mu$-measure.
\vspace{0.2cm}

\noindent\emph{\bf Step 2.} The second part of the proof consists on using the cover above to construct a global  version of Theorem \ref{zhanglemma}. The localization that enables us to use the G{\aa}rding inequality on the rectangles built above will lead to having an \textit{error} term, that will vanish when we consider suitable sequences of variations. 

Let $\lambda \mathcal{R}_j$ denote the dilation of $\mathcal{R}_j$ with respect to the center of the rectangle by a factor of $\lambda\in [1,2]$. Let also $\Omega_{\lambda\mathcal{R}_j}:=\Omega\cap\lambda\mathcal{R}_j$. We will show that 
\begin{align}\label{eq8}
&\frac{c_0}{2}\underset{\Omega}{\int}|S(\nabla\varphi(x))|^2 \dx-c\sum_{j\in J}\underset{\Omega_{\lambda\mathcal{R}_j}-\Omega_{\mathcal{R}_j}}{\int}\left(|S(\nabla\varphi(x))|^2+\left|S\left(\frac{\varphi(x)}{\lambda-1}\right)\right|^2\right)\dx\notag\\
\leq& \underset{\Omega}{\int}(F(x, u_0(x)+\varphi(x),\nabla u_0(x)+\nabla\varphi(x))-F(x)+ C |\varphi(x)|^2)\dx\notag\\
=&\underset{\Omega}{\int}\left(G(x,\varphi(x),\nabla\varphi(x))+ C |\varphi(x)|^2\right)\dx
\end{align}
for all   $\lambda\in (1,2)$ and all $\varphi\in \overline{\mathrm{Var}(\Omega,\R^N)^{\W^{1,p}}}$ with $||\varphi||_{\LL^\infty(\Omega,\R^N)} < \delta$ and $\delta>0$ as in Theorem \ref{zhanglemma}.\vspace{0.2cm}

Consider the cover for $\Omega$ built in the previous step and, for each $j\in J$ and any $\lambda\in(1,2)$, take cut-off functions $\rho_j\in {\rm C}^1_c(2\mathcal{R}_j)$ such that $\mathbbm{1}_{\mathcal{R}_j}\leq \rho_j\leq\mathbbm{1}_{\lambda\mathcal{R}_j}$ and $|\nabla\rho_j|\leq\frac{c}{\lambda-1}$ with $c>0$ a constant independent of $j$. Note that the rectangles $\lambda\mathcal{R}_j$ have bounded overlap since there is a finite number of them. 

Additionally, if $\varphi\in \overline{\mathrm{Var}(\Omega,\R^N)^{\W^{1,p}}}$, then $\rho_j\varphi\in \overline{\mathrm{Var}(\Omega_{\lambda\mathcal{R}_j},\R^N)^{\W^{1,p}}}$ and, since $2\mathcal{R}_j\subseteq \Omega(x_0,R)$ for some neighbourhood adequate to guarantee \eqref{eq:minimality} from Theorem \ref{zhanglemma}, then

\begin{align*}
\underset{\Omega_{\lambda\mathcal{R}_j}}{\int}\frac{c_0}{2}|S(\nabla(\rho_j\varphi))|^2\dx\leq& \underset{\Omega_{\lambda\mathcal{R}_j}}{\int}\left[F(x,u_0+\rho_j\varphi,\nabla u_0+\nabla(\rho_j\varphi)) -F(x)\right]\dx
+ \underset{\Omega_{\lambda\mathcal{R}_j}}{\int} C|\rho_j\varphi|^2\dx.
\end{align*}

Since $u_0$ is an $F$-extremal, this also implies that 
\begin{equation*}
\frac{c_0}{2}\underset{\Omega_{\lambda\mathcal{R}_j}}{\int}|S(\nabla(\rho_j\varphi))|^2\dx\leq\underset{\Omega_{\lambda\mathcal{R}_j}}{\int}\left[G(x,\rho_j\varphi,\nabla(\rho_j\varphi)) + C|\rho_j\varphi|^2\right]\dx.
\end{equation*}

But $\rho_j=1$ on $\Omega_{\mathcal{R}_j}$ and, hence, 
\begin{align*}
&\frac{c_0}{2}\underset{\Omega_{\mathcal{R}_j}}{\int}|S(\nabla\varphi)|^2\dx+\frac{c_0}{2}\underset{\Omega_{\lambda\mathcal{R}_j}-\Omega_{\mathcal{R}_j}}{\int}|S(\nabla(\rho_j\varphi))|^2\dx\\
\leq&\underset{\Omega_{\mathcal{R}_j}}{\int}\left[G(x,\varphi,\nabla\varphi)+C|\varphi|^2\right]\dx+\underset{\Omega_{\lambda\mathcal{R}_j}-\Omega_{\mathcal{R}_j}}{\int}\hspace*{-5mm}\left[G(x,\rho_j\varphi,\nabla(\rho_j\varphi))+C|\rho_j\varphi|^2|\right]\dx.
\end{align*}

By Lemma \ref{LemmagrowthG} (a), we find that for each $x\in\overline{\Omega}$, $y\in\R^N$ and $z\in\R^{N\times d}$, the function $G$ satisfies
\begin{equation}\label{Ggrowth}
|G(x,y,z)|\leq  C(y) \left(|S(y)|^2+|S(z)|^2\right)
\end{equation}
for some locally bounded $C(y)$. Noting that $\varphi$ is bounded, using \eqref{Ggrowth} and after adding up the previous inequalities over $j$, we obtain
\begin{align*}
\frac{c_0}{2}\underset{\Omega}{\int}|S(\nabla\varphi)|^2 \dx \leq &\frac{c_0}{2}\underset{\Omega}{\int}|S(\nabla\varphi)|^2 \dx+\frac{c_0}{2}\sum_{j\in J}\underset{\Omega_{\lambda\mathcal{R}_j}-\Omega_{\mathcal{R}_j}}{\int}|S(\nabla(\rho_j\varphi))|^2\dx\\
\leq& \underset{\Omega}{\int}(F(x,u_0+\varphi,\nabla u_0+\nabla\varphi)-F(x) + C|\varphi|^2)\dx\\
&\qquad+c\sum_{j\in J}\underset{\Omega_{\lambda\mathcal{R}_j}-\Omega_{\mathcal{R}_j}}{\int}\left[|S(\rho_j\varphi)|^2+|S(\nabla(\rho_j\varphi))|^2+C|\rho_j\varphi|^2|\right]\dx.\\
\leq & \underset{\Omega}{\int}(F(x,u_0+\varphi,\nabla u_0+\nabla\varphi)-F(x) + C|\varphi|^2)\dx\\
&\qquad+c\sum_{j\in J}\underset{\Omega_{\lambda\mathcal{R}_j}-\Omega_{\mathcal{R}_j}}{\int}\left[|S(\varphi)|^2+|S(\nabla\varphi)|^2+\left|S\left(\frac{\varphi}{\lambda-1}\right)\right|^2+C|\varphi|^2|\right]\dx,
\end{align*}
since $|\rho_j|\leq 1$. Hence, \eqref{eq8} follows because $0 < \lambda-1 <1$ and $|S(\varphi)|^2\leq 2 |\varphi|^2$ as $\|\varphi\|_{\infty} < \delta$.
\vspace{0.2cm}

\noindent\emph{\bf Step 3.} To conclude the proof of the Proposition, we take a sequence $(h_k)\subseteq \overline{{\rm Var}(\Omega,\R^N)^{\W^{1,p}}}\cap \LL^{\infty}(\Omega,\R^N)$ such that:
\begin{itemize}
\item $\| h_k\|_{\LL^\infty}\overset{k\rightarrow\infty}{\longrightarrow}0$;
\item $a_k^{-1}S( h_k) \overset{k\rightarrow\infty}{\longrightarrow}0$  in $\LL^2(\Omega)$;
\item $a_k^{-1}S(\nabla h_k)$ is bounded in $\LL^2(\Omega)$. 
\end{itemize}
Then, for $k\in\N$ large enough, we can apply \eqref{eq8} with $\varphi=h_k$ to obtain, after dividing by $a_k^2$, that
\begin{align}
&\frac{c_0}{2} a_k^{-2}\underset{\Omega}{\int}|S(\nabla h_k(x))|^2\dx-c\sum_{j\in J}\underset{\Omega_{\lambda\mathcal{R}_j}-\Omega_{\mathcal{R}_j}}{\int}a_k^{-2} \left(|S(\nabla h_k(x)|^2+a_k^{-2}\left|S\left(\frac{h_k(x)}{\lambda-1}\right)\right|^2\right)\dx\notag\\
\leq& \, a_k^{-2}\underset{\Omega}{\int}(F(x, u_0(x)+h_k,\nabla u_0(x)+\nabla h_k (x))-F(x)+ C |h_k(x)|^2)\dx\notag\\
=&\,a_k^{-2}\underset{\Omega}{\int}\left(G(x, h_k(x),\nabla h_k(x))+ C |h_k(x)|^2\right)\dx.\label{eq:forGlobalGarding}
\end{align}
Next, since $a_k^{-1}S(\nabla h_k)$ is bounded in $\LL^2(\Omega)$, for a subsequence that we do not relabel we may assume that  $a_k^{-2}|S(\nabla h_k)|^2\mathcal{L}^d\overset{*}{\rightharpoonup}\mu$ in ${\rm C}^0\left(\overline{\Omega}\right)^*\cong\mathcal{M}\left(\overline{\Omega}\right)$, where $\mathcal L^d$ denotes the $d$-dimensional Lebesgue measure. Furthermore, given that $a_k^{-1}S(h_k)\overset{k\rightarrow\infty}{\longrightarrow}0$ in $\LL^2$, then for every $\lambda\in (0,1)$, 
\begin{equation*}
\limsup_{k\to\infty} \sum_{j\in J}\underset{\Omega_{\lambda\mathcal{R}_j}-\Omega_{\mathcal{R}_j}}{\int}\left(a_k^{-2}|S(\nabla h_k(x)|^2+a_k^{-2}\left|S\left(\frac{h_k(x)}{\lambda-1}\right)\right|^2\right)\dx \leq \sum_{j\in J}\mu\left(\overline{\Omega}\cap(\overline{\lambda\mathcal{R}_j}-\mathcal{R}_j)\right). 
\end{equation*}
Whereby, taking $\liminf$ in \eqref{eq:forGlobalGarding}, we obtain that
\begin{align*}
&\liminf_{k\rightarrow\infty}\frac{c_0}{2}a_k^{-2}\underset{\Omega}{\int}|S(\nabla h_k(x))|^2\dx-c\sum_{j\in J}\mu\left(\overline{\Omega}\cap(\overline{\lambda\mathcal{R}_j}-\mathcal{R}_j)\right)\notag\\
\leq &\liminf_{k\rightarrow\infty}\,a_k^{-2}\underset{\Omega}{\int}G(x, h_k(x),\nabla h_k(x))\dx.
\end{align*}
Note that for the last inequality above we are using that $\|h_k\|_{\LL^2}\leq \|S(h_k)\|_{\LL^2}$. 

Finally, by sending $\lambda \searrow 1$ and using \eqref{eq:mu_sing0}, we conclude that
\begin{align*}
\liminf_{k\rightarrow\infty}\,\frac{c_0}{2}a_k^{-2}\underset{\Omega}{\int}|S(\nabla h_k(x))|^2\dx =&\,\liminf_{k\rightarrow\infty}\frac{c_0}{2}a_k^{-2}\underset{\Omega}{\int}|S(\nabla h_k(x))|^2\dx-c\sum_{j\in J}\mu\left(\overline{\Omega}\cap\partial\mathcal{R}_j\right)\notag\\
\leq &\liminf_{k\rightarrow\infty}\,a_k^{-2}\underset{\Omega}{\int}G(x,h_k(x),\nabla h_k(x))\dx.
\end{align*}

\end{proof}

A fundamental tool in the proof of the Sufficiency Theorem (Theorem \ref{thm:sufficiency}) is a decomposition result that finds its origins in the Decomposition Lemma established by Kristensen \cite{KristDecomp,KristensenLSC} and, by other means, by Fonseca-M\"uller-Pedregal in \cite{FonsMullPed}.  This result allows, up to subsequences, the splitting of a weakly converging sequence into an oscillating and a concentrating part. We enunciate here Theorem \ref{lemma:decomposition} concerning a variant of the Decomposition Lemma that enables us to split simultaneously the normalisations in $\W^{1,2}$ and $\W^{1,p}$, respectively, of a given sequence. This version of the Decomposition Lemma was established in \cite{GM09}, where its proof can be found. We remark that this is based on the Lipschitz truncation strategy followed in \cite{FonsMullPed}, while Kristensen's proof uses the Helmholtz Decomposition Theorem.

\begin{theorem}\label{lemma:decomposition}
Let $\Omega\subseteq\R^d$ be a bounded Lipschitz domain and $p\in[2,\infty)$. Let $(\psi_k)\subseteq\varA$ such that $\psi_k\rightharpoonup \psi$ in $\W^{1,2}(\Omega,\R^N)$ and assume that $(\eta_k)$ is a sequence in $(0,1]$ such that $\eta_k\psi_k$ is bounded in $\W^{1,p}(\Omega,\R^N)$. If $p=2$, assume that $\eta_k=1$. Further, suppose that $\alpha_k>0$, $\alpha_k\to0$ and that $\alpha_k\psi_k\to0$ uniformly in $\Omega$. Then, there exist a subsequence of $(\psi_k)$ (not relabelled), sequences $(g_k)\subseteq -\psi+\W^{1,\infty}(\Omega,\R^N)$, $(b_k)\subseteq \W^{1,\infty}(\Omega,\R^N)$ with $\psi+g_k(x)=b_k(x) = 0$ for $x\in\Gamma_D$ and $R_k\subset\Omega$ such that 
\begin{itemize}
\item[\textup{(a)}] $\psi_k=\psi + g_k + b_k$;
\item[\textup{(b)}] $g_k\rightharpoonup 0$ and $b_k\rightharpoonup 0$ in $\W^{1,2}(\Omega,\R^N)$;
\item[\textup{(c)}] for all $x\in\Omega\backslash R_k$, $\psi_k=\psi+g_k$ and $\nabla\psi_k=\nabla\psi+\nabla g_k$;
\item[\textup{(d)}] $\mathcal{L}^d(R_k)\longrightarrow 0$ and, hence, $\nabla b_k\rightarrow 0$ in measure;
\item[\textup{(e)}] $(|\nabla g_k|^2)$ and $(|\eta_k\nabla g_k|^p)$ are both equiintegrable;
\item[\textup{(f)}] $\alpha_k(\psi + g_k)\rightarrow 0$ and $\alpha_k b_k\rightarrow 0$ uniformly in $\Omega$ and
\item[\textup{(g)}] $\eta_k b_k\rightharpoonup 0$ in $\W^{1,p}(\Omega,\R^N)$.
\end{itemize}
\end{theorem}

We note that since $b_k\in\W^{1,\infty}(\Omega,\R^N)$ and $b_k = 0$ on $\Gamma_D$, then also $b_k\in\overline{\varA^{\W^{1,p}}}$. The following lemma will also play a crucial role in the proof of Theorem \ref{thm:sufficiency} and, though simple, we present its proof here.
\begin{lemma}\label{lemma:F''QC}
Let $W\colon\R^{N\times d}\rightarrow\R$ and $z_0\in\R^{N\times d}$. Suppose that $W$ is of class $C^2$ in a neighbourhood of $z_0$ and that it is strongly quasiconvex (in the interior) at $z_0$, i.e. for all $\varphi\in {\rm C}^1_c({B},\R^N)$,
\[
\int_B\left[W(z_0+\nabla\varphi)-W(z_0)\right]\dx\geq c_0\int_B |S(\nabla\varphi)|^2\dx.
\]
Then, for every $z_0\in\R^{N\times d}$ and every $\varphi\in {\rm C}^1_c(B,\R^N)$, it holds that
\begin{equation*}
2c_0\int_\Omega|\nabla\varphi|^2\dx\leq\int_\Omega W_{zz}(z_0)\nabla\varphi\cdot\nabla\varphi\dx,
\end{equation*}
i.e. $W_{zz}(z_0)[\cdot,\cdot]-2c_0|\cdot|^2$ is quasiconvex.
\end{lemma}

\begin{proof}
By the quasiconvexity at $z_0$ it follows that, for every $\varphi\in {\rm C}^1_c(B,\R^N)$, $t=0$ minimises the real valued function
\begin{equation}
J(t):=\int_B\left[W(z_0+t\nabla\varphi)-W(z_0)-c_0|S(t\nabla\varphi)|^2\right]\dx.
\end{equation}
Hence, 
\begin{equation}
0\leq J''(0)=\int_B\left[W_{zz}(z_0)\nabla\varphi\cdot\nabla\varphi-2c_0|\nabla\varphi|^2\right)\dx.
\end{equation}
\end{proof}

\begin{remark}\label{remark:F''QC}
For $F$ and $u_0$ as in Theorem \ref{thm:sufficiency}, Lemma \ref{lemma:F''QC} says that condition $(III)$ on the strong quasiconvexity in the interior implies that, for each $x\in\overline{\Omega}$ and any $a(x)\in\R^{N\times d}$, the function $H(x,z) = F_{zz}(x)z\cdot z - 2c_0|z - a(x)|^2$ is quasiconvex at every $z\in\R^{N\times d}$.
\end{remark}
\subsection{Proof of the Sufficiency Theorem}
The remainder of this Section is devoted to the proof of the Sufficiency Theorem (Theorem \ref{thm:sufficiency}).
\begin{proof}
We prove the result arguing by contradiction. Suppose that the theorem does not hold. Then, we can find a sequence $(\varphi_k)\subseteq \mathrm{Var}(\Omega,\R^N)$ such that $\|\varphi_k\|_{\LL^\infty(\Omega,\R^N)}\rightarrow 0$ and
\begin{equation}\label{forcontGM}
\underset{\Omega}{\int} F(x,u_0(x)+\varphi_k(x),\nabla u_0(x) +\nabla\varphi_k(x))\dx<\underset{\Omega}{\int} F(x)\dx
\end{equation}
for all $k\in\N$. 
As in Lemma \ref{LemmagrowthG}, we use Taylor's Theorem and define 
\begin{align*}
G(x,y,z):=&F(x,u_0(x)+y,\nabla u_0(x)+z)-F(x)-F_y(x)\cdot y-F_z(x)\cdot z\notag\\
=&\int_0^1 (1-t)L(x, u_0(x)+ty,\nabla u_0(x)+tz)[(y,z),(y,z)]\dt,
\end{align*}
where
\begin{align*}
L(x,v,w)[(y,z),(y,z)]=& F_{yy}(x,v,w)y\cdot y+ 2F_{yz}(x,v,w)y\cdot z+F_{zz}(x,v,w)z\cdot z.
\end{align*}

Note that, since $u_0$ is an $F$-extremal, for every $k\in\N$ it holds that
\begin{align}
&\underset{\Omega}{\int} G(x,\varphi_k,\nabla\varphi_k)\dx\notag\\
=&\underset{\Omega}{\int}\int_0^1(1-t) L(x, u_0(x)+t\varphi_k(x),\nabla u_0(x)+t\nabla\varphi_k(x))[(\varphi_k,\nabla\varphi_k),(\varphi_k,\nabla\varphi_k)]\dt\dx\notag\\
=&\underset{\Omega}{\int} \bigl(F(x,u_0(x)+\varphi_k(x),\nabla u_0(x)+\nabla\varphi_k(x))-F(x)\bigr)\dx<0.
\label{eqGpsi}
\end{align} 

This inequality suggests the underlying idea of the proof which is to exploit the strong positivity of the second variation to obtain a contradiction. Indeed, by a normalisation argument we show that one can construct a sequence of variations $(\psi_k)$ suitable for this purpose.

We remark that, for simplicity, we will use the notation $\Omega(x,r):=\Omega\cap Q(x,r)$, where $Q(x,r)$ is the cube with side length $2r$ rather than the ball $B(x,r)$. Note that, due to Remark \ref{remark:cubes}, Theorem \ref{zhanglemma} remains valid if we use cubes instead of balls. We divide the proof into several steps.\vspace{0.2cm}

\noindent\emph{\bf Step 1.} In this step we will show that the coercivity assumption [H2] reduces the problem to the case of $\W^{1,p}$-local minimisers. In particular, we show the following. \vspace{0.2cm}

\noindent{\bf Claim 1.} Let $\gamma_k:= ||S(\nabla\varphi_k)||_{\LL^2}$, $\alpha_k:= \|\nabla\varphi_k\|_{\LL^2}$ and $\beta_k:= (2|\Omega|)^{\frac{1}{2}-\frac{1}{p}}\|\nabla\varphi_k\|_{\LL^p}$. We claim that $\gamma_k\to0$ and consequently, $\alpha_k\to0$ and $\beta_k\to 0$. Moreover, the sequence of variations $(\varphi_k)$ is such that 
\begin{equation}\label{betapbyalpha2bded}
0\leq\underset{k\in\N}{\sup}\frac{\beta_k^p}{\alpha_k^2}=\Lambda<\infty
\end{equation}
for some real number $\Lambda>0$.\vspace{0.2cm}

\noindent\emph{\bf Proof of Claim 1.} By virtue of assumption [H2] and \eqref{forcontGM}, $(\varphi_k)$ is uniformly bounded in $\W^{1,p}(\Omega,\R^N)$ and must therefore converge weakly to $0$, since it converges strongly to $0$ in $\LL^\infty(\Omega,\R^N)$. 

Note that $\gamma_k>0$ for all $k\in\N$ and $(\gamma_k)$ is a bounded sequence because $p\geq 2$. Therefore, up to a subsequence that we do not relabel, $\gamma_k\to\gamma\geq0$. Next, setting $\alpha_k=1$ and $h_k = \varphi_k$ in Proposition \ref{global_garding} and using \eqref{eqGpsi}, we infer that

\begin{equation*}
0 \leq \frac{c_0}{2}\gamma = \liminf_{k\rightarrow\infty}\frac{c_0}{2}\underset{\Omega}{\int}|S(\nabla \varphi_k(x))|^2\dx 
\leq \liminf_{k\rightarrow\infty}\underset{\Omega}{\int}G(x,\varphi_k(x),\nabla \varphi_k(x))\dx \leq 0.
\end{equation*}

Hence, $\gamma = 0$ and $\gamma_k\to0$. Regarding the boundedness of $\beta_k^p/\alpha_k^2$, this is trivial if $p=2$. If $p>2$, from the coercivity assumption (H2) applied to $\varphi_k$ it follows, after dividing by $\alpha_k^2$, that for every $k\in\N$, 
\begin{equation*}
c_1\frac{\beta_k^p}{\alpha_k^2}-c_2\leq\alpha_k^{-2}\underset{\Omega}{\int}\left[F(x,u_0(x)+\varphi_k(x),\nabla u_0(x) + \nabla\varphi_k(x))-F(x)\right]\dx<0.
\end{equation*}
This proves that the sequence $\left( \frac{\beta_k^p}{\alpha_k^2}\right)$ is bounded and the claim follows. We remark that, similarly to \cite{GM09}, assumption [H2] was required precisely to reduce the problem to that of $\W^{1,p}$-local minimisers.
\\

\noindent\emph{\bf Step 2.} Note that, by Claim 1, 
\[
\int_{\Omega}G(x,\varphi_k,\nabla\varphi_k)\dx \to 0
\]
and a contradiction cannot be reached in this way. Instead, we define the normalised sequence of variations $\psi_k = \alpha_k^{-1}\varphi_k\in \mathrm{Var}(\Omega,\R^N)$. The sequence $\psi_k$ may be bounded in $\W^{1,2}$ but, unlike $\varphi_k$, will concentrate and fail to converge strongly in $\W^{1,2}$. We first decompose the sequence $\psi_k$ into an oscillating and a concentrating part using Lemma \ref{lemma:decomposition}.

To this end, let $\eta_k:=\frac{\alpha_k}{\beta_k}$ and note that, by H\"{o}lder's inequality,
\begin{equation*}
\alpha_k= \|\nabla\varphi_k\|_{\LL^2}\leq (2|\Omega|)^{\frac{1}{2}-\frac{1}{p}}\|\nabla\varphi_k\|_{\LL^p}=\beta_k.
\end{equation*}
Therefore, $\eta_k=\frac{\alpha_k}{\beta_k}\in(0,1]$ for every $k\in\N$. Also, it is clear that $\alpha_k\psi_k=\varphi_k\to0$ uniformly and, to apply the Decomposition Lemma, we need only observe that $(\eta_k\psi_k)$ is bounded in $\W^{1,p}(\Omega,\R^N)$, because $\underset{\Omega}{\int}|\eta_k\nabla\psi_k|^p=\beta_k^{-p}\underset{\Omega}{\int}|\nabla\varphi_k|^p=(2|\Omega|)^{1-\frac{p}{2}}$. Then, up to a subsequence that we do not relabel, there exist sequences $(g_k)\subseteq-\psi+\W^{1,\infty}(\Omega,\R^N)$ and $(b_k)\subseteq \W^{1,\infty}(\Omega,\R^N)$ with $\psi+g_k=b_k=0$ on $\Gamma_D$, as well as sets $(R_k)\subseteq\Omega$
such that:
\begin{itemize}
\item $\psi_k=\psi + g_k + b_k$;
\item $g_k\rightharpoonup 0$ and $b_k\rightharpoonup 0$ in $\W^{1,2}(\Omega,\R^N)$;
\item for all $x\in\Omega\backslash R_k$, $\psi_k=\psi+g_k$ and $\nabla\psi_k=\nabla\psi+\nabla g_k$;
\item $\mathcal{L}^d(R_k)\longrightarrow 0$ and, hence, $\nabla b_k\rightarrow 0$ in measure;
\item $(|\nabla g_k|^2)$ and $(|\eta_k\nabla g_k|^p)$ are both equiintegrable;
\item $\alpha_k(\psi + g_k)\rightarrow 0$ and $\alpha_k b_k\rightarrow 0$ uniformly in $\Omega$ and
\item $\eta_k b_k\rightharpoonup 0$ in $\W^{1,p}(\Omega,\R^N)$.
\end{itemize}

Next, as in \cite{GM09}, we proceed to prove an orthogonality principle between the oscillating and concentrating parts of $\psi_k$ implying that they act on the functional independently. We show the following:\vspace{0.2cm}

\noindent{\bf Claim 2.} As $k\to\infty$, it holds that
\[
\alpha_k^{-2} \int_{\Omega}\left[ G(x,\alpha_k\psi_k,\alpha_k \nabla \psi_k) - G(x,\alpha_kb_k,\alpha_k \nabla b_k) - G(x,\alpha_k(\psi+ g_k),\alpha_k (\nabla\psi+ \nabla g_k))\right]\dx \to 0.
\]
In particular,
\begin{align*}
&\liminf_{k\to\infty}\alpha_k^{-2} \int_{\Omega}\left[ G(x,\alpha_kb_k,\alpha_k \nabla b_k) + G(x,\alpha_k(\psi+ g_k),\alpha_k (\nabla\psi + \nabla g_k))\right]\dx \\
 = &\liminf_{k\to\infty} \alpha_k^{-2} \int_{\Omega} \left|G(x,\alpha_k\psi_k,\alpha_k \nabla \psi_k)\right|\dx \\
  =  & \liminf_{k\to\infty} \alpha_k^{-2} \int_{\Omega} G(x,\varphi_k, \nabla\varphi_k)\dx \leq 0.
\end{align*}

\noindent\emph{\bf Proof of Claim 2.} Recall that, by the Decomposition Lemma, $\psi_k = \psi + g_k$ for $x\in \Omega\backslash R_k$ and that $\mathcal{L}^d(R_k)\to 0$. Then, 
\begin{align*}
& \limsup_k \left|\alpha_k^{-2} \int_{\Omega}\left[ G(x,\alpha_k\psi_k,\alpha_k \nabla \psi_k) - G(x,\alpha_kb_k,\alpha_k \nabla b_k) - G(x,\alpha_k(\psi+ g_k),\alpha_k (\nabla\psi +\nabla g_k)\right]\dx\right|\\ 
&\leq \limsup_k  \int_{R_k} |f_k(x)|\dx + \limsup_k \alpha_k^{-2}\int_{R_k} \left|G(x,\alpha_k(\psi+ g_k),\alpha_k (\nabla\psi+ \nabla g_k)\right|\dx\\
& =: I+II,
\end{align*}
where
\[
f_k(x):= \alpha_k^{-2}\left[ G(x,\alpha_k\psi_k,\alpha_k \nabla \psi_k) - G(x,\alpha_kb_k,\alpha_k \nabla b_k)\right].
\]
We show that $I=II = 0$ by proving the equiintegrability of the integrands. To prove that $II=0$, simply note that by Lemma \ref{LemmagrowthG} (a)
\begin{align*}
&\alpha_k^{-2} \left|G(x,\alpha_k(\psi+ g_k),\alpha_k (\nabla\psi + \nabla g_k))\right| \\
&\quad  \leq \alpha_k^{-2}C(\alpha_k (\psi+g_k))\left[ |S(\alpha_k (\psi+g_k))|^2 + |S(\alpha_k(\nabla \psi+ \nabla g_k))|^2\right]\\
&\quad \leq C(\alpha_k(\psi+g_k)) \left[ |\psi + g_k|^2 + \alpha_k^{p-2}|\psi + g_k|^p + |\nabla\psi+\nabla g_k|^2 + \alpha_k^{p-2}|\nabla\psi+\nabla g_k|^p\right].
\end{align*}
However, by the Decomposition Lemma, $\alpha_k(\psi+ g_k)\to0$ uniformly and $(|g_k|^2)$ as well as $(|\nabla g_k|^2)$ are equiintegrable. Also, by the boundedness of $\left(\frac{\beta_k^p}{\alpha_k^2}\right)$ and the fact that
\begin{equation*}
\alpha_k^{p-2}|\nabla g_k|^p= \frac{\beta_k^p}{\alpha_k^2}\eta_k^p|\nabla g_k|^p, 
\end{equation*}
we deduce that $(\alpha_k^{p-2}|\nabla\psi+\nabla g_k|^p)$ is equiintegrable and so is $(\alpha_k^{p-2}|\psi+g_k|^p)$. Hence, $II = 0$.
\vspace{0.2cm}

\noindent
To prove that $I=0$, we prove the equiintegrability of $(f_k)$. Note that by Lemma \ref{LemmagrowthG} (a),
for every $y,\hat{ y}\in\R^N$, ${z},\hat{ z}\in\R^{N\times d}$ and for every $x\in\overline{\Omega}$,
\begin{equation*}
|G(x,y,z)-G(x,\hat{ y},\hat{ z})|\leq C(y,\hat y)\left[A_{p-1}(y,z,\hat y,\hat z)|z-\hat z|+A_p(y,z,\hat y,\hat  z)|y - \hat y|\right].
\end{equation*}
Then, by Young's inequality, we obtain that for any $\varepsilon>0$, there exists a constant $C_\varepsilon$ such that 
\begin{align*}
|f_k|\leq &\, C\alpha_k^{-1}\left( A_{p-1}(\alpha_k\psi_k,\alpha_k b_k,\alpha_k\nabla\psi_k,\alpha_k\nabla b_k)|\nabla\psi+\nabla g_k|\right)\\
&+C\alpha_k^{-1} \left(A_{p}(\alpha_k\psi_k,\alpha_k b_k,\alpha_k\nabla\psi_k,\alpha_k\nabla b_k) \right) |\psi+ g_k|\\
\leq & C\left(|\psi_k| + |b_k|+|\nabla\psi_k|+|\nabla b_k|+\alpha_k^{p-2}|\nabla\psi_k|^{p-1}+\alpha_k^{p-2}|\nabla b_k|^{p-1}\right)|\nabla\psi+\nabla g_k| \\
&\quad + C\left(|\psi_k| + |b_k|+|\nabla\psi_k|+|\nabla b_k|+\alpha_k^{p-1}|\nabla\psi_k|^{p}+\alpha_k^{p-1}|\nabla b_k|^{p}\right)|\psi+g_k|\\
\leq & \eps C \left(|\psi_k|^2 + |b_k|^2+|\nabla\psi_k|^2+|\nabla b_k|^2 +\alpha_k^{p-2}|\nabla\psi_k|^{p}+\alpha_k^{p-2}|\nabla b_k|^{p} \right) \\
&\quad + C_\eps \left(|\nabla\psi + \nabla g_k|^2+\alpha_k^{p-2}|\nabla\psi + \nabla g_k|^p\right)\\
&\quad+\eps C\left(|\psi_k|^2 + |b_k|^2+|\nabla\psi_k|^2+|\nabla b_k|^2\right)\\
&\quad+ C_\eps|\psi+g_k|^2+\left(\alpha_k^{p-2}|\nabla\psi_k|^{p}+\alpha_k^{p-2}|\nabla b_k|^{p}\right)|\alpha_k(\psi+g_k)|=:\sum_{i=1}^5e_i(x).
\end{align*}
Above, $C$ is a constant (varying from line to line) since $\alpha_k\psi_k$ and $\alpha_k b_k$ are uniformly bounded. Observe that $\psi_k$, $b_k$ are both bounded in $\W^{1,2}(\Omega,\R^N)$. Also,
\[
\alpha_k^{p-2} |\nabla\psi_k|^{p}  = \frac{\beta_k^p}{\alpha_k^2} |\eta_k\nabla\psi_k|^p
\]
is bounded in $\LL^1(\Omega,\R^N)$, since $\left(\frac{\beta_k^p}{\alpha_k^2}\right)$ is bounded by (\ref{betapbyalpha2bded}) and $(\eta_k\psi_k)$ is bounded in $\W^{1,p}(\Omega,\R^N)$. Similarly, the same is true for $\alpha_k^{p-2}|\nabla b_k|^p$ by the boundedness of $(\eta_k b_k)$ in $\W^{1,p}(\Omega,\R^N)$.  Hence, for any set $A\subset\Omega$ we have
\[
\int_A |f_k(x)|\dx \leq \eps C + \int_A e_2(x)\dx  + \int_A e_4(x)\dx + \int_A e_5(x)\dx.
\]
As with term $II$, since $\alpha_k(\psi+ g_k)\to0$ uniformly, $e_5$ is equiintegrable. Also, by the boundedness of $\left(\frac{\beta_k^p}{\alpha_k^2}\right)$ and the fact that
\begin{equation*}
\alpha_k^{p-2}|\nabla g_k|^p= \frac{\beta_k^p}{\alpha_k^2}\eta_k^p|\nabla g_k|^p, 
\end{equation*}
we deduce that $(\alpha_k^{p-2}|\nabla g_k|^p)$ is equiintegrable and, hence, so is $e_2$. Finally, since $(|\psi+g_k|^2)$ is also equiintegrable, the same holds for $(f_k)$. This concludes the proof of Claim 2 and the orthogonality principle.\vspace{0.2cm}

\noindent{\bf Step 3.} In the same spirit as \cite{GM09}, we now prove that the quasiconvexity conditions prevent the concentrating part of the sequence $\psi_k$, i.e. $b_k$, to lower the energy. We remark, however, that our strategy to achieve this differs from the original one in \cite{GM09} on that it fully relies on the G{\aa}rding-type inequality from Proposition \ref{global_garding}. Whereby, under the smoothness assumption on the extremal, it persists as a natural consequence of the quasiconvexity conditions. Specifically, we show the following:\vspace{0.2cm}

\noindent{\bf Claim 3.} 
\[
\liminf_k \alpha_k^{-2} \int_{\Omega} G(x,\alpha_kb_k,\alpha_k \nabla b_k)\dx \geq 0.
\]
This implies, in turn, that
\[
\liminf_k  \int_{\Omega} f_k(x)\dx \leq \liminf_k \alpha_k^{-2} \int_{\Omega} G(x,\varphi_k,\nabla\varphi_k)\dx \leq 0.
\]
\noindent{\bf Proof of Claim 3.} In virtue of Claim 1 and since $\alpha_k^{p-2}=\frac{\beta_k^p}{\alpha_k^2}\eta_k^p$, the first part of the claim is a direct consequence of applying Proposition \ref{global_garding} with $h_k=\alpha_kb_k$ and $a_k=\alpha_k$. The second part follows from this, since then
\begin{align*}
\liminf_k  \int_{\Omega} f_k(x)\dx \leq & \liminf_k  \int_{\Omega} f_k(x)\dx + \liminf_k \alpha_k^{-2} \int_{\Omega} G(x,\alpha_kb_k,\alpha_k \nabla b_k)\dx\\
\leq &  \liminf_k \alpha_k^{-2} \int_{\Omega} G(x,\varphi_k,\nabla\varphi_k)\dx \leq 0.
\end{align*}

\vspace{0.2cm}

\noindent{\bf Step 4.} Step 3 has established that any reduction of the energy must come from the purely oscillating sequence $(\psi + g_k)$. In this step we show that this term can be controlled from below by a \textit{Young measure version} of the second variation, which must be nonpositive due to the inequality proved in Step 3. This will contradict the positivity of the second variation. \vspace{0.2cm}

\noindent{\bf Claim 4.} Let $(\nu_x)_x$ be the Young measure generated by the sequence $(\nabla\psi_k)$, which is bounded in $\LL^2$. Then, 
\[
\frac12 \int_\Omega\left[F_{yy}(x)\psi\cdot\psi + 2 F_{yz}(x)\psi\cdot\nabla\psi + \int_{\R^{N\times d}} F_{zz}(x) z\cdot z \dv\nu_x(z)\right]\dx \leq \liminf_k  \int_{\Omega} f_k(x)\dx.
\]
Hence, by Claim 3, we also have that
\[
\frac12 \int_\Omega\left[F_{yy}(x)\psi\cdot\psi + 2 F_{yz}(x)\psi\cdot\nabla\psi + \int_{\R^{N\times d}} F_{zz}(x) z\cdot z \dv\nu_x(z)\right]\dx \leq 0.
\] 

\noindent{\bf Proof of Claim 4.} Recall that in the proof of Claim 2 it was established that $(f_k)$ is equiintegrable. Now, let $\varepsilon>0$. Since $(\nabla\psi_k)$ is measure-tight\footnote{A sequence $(u_k)$ is measure-tight if $\lim_{t\to\infty} \sup_k \mathcal L^d(\{x\in\Omega\st |u_k(x)|>t\}) = 0$.} and $\nabla b_k\rightarrow 0$ in measure, we can take $m_\varepsilon>0$ large enough so that, for every $m\geq m_\varepsilon$, 
\begin{equation*}
\underset{\{|\nabla\psi_k|\geq m\}\cup\{|\nabla b_k|\geq m\}}{\int}|f_k(x)|\dx<\varepsilon
\end{equation*}
for all $k\in\N$. Hence, for all $m\geq m_\varepsilon$, 
\begin{equation}\label{eq:fkminuseps}
\underset{\{|\nabla\psi_k|< m\}\cap\{|\nabla b_k|< m\}}{\int}f_k(x)\dx-\varepsilon<\underset{\Omega}{\int} f_k(x)\dx.
\end{equation}
Also note that, since the Young measure $\nu_x$ has finite second moment (see \cite[Lemma 4.3]{RindlerCV}) and $x\mapsto F_{zz}(x)$ is uniformly bounded, the Dominated Convergence Theorem implies that, by taking $m_\eps$ larger if necessary, 
\begin{equation*}
\left|\underset{\Omega}{\int}\underset{\R^{N\times d}}{\int}F_{zz}(x)z\cdot z\mathbbm{1}_{\R^{N\times d}\backslash\overline{B(0,m)}}(z)\dv\nu_x(z)\dx\right|<\varepsilon\mbox{ \,\,\,for all }m\geq m_\varepsilon.
\end{equation*}
Then, for $m\geq m_\eps$,
\begin{align}\label{eq:pluseps}
&\int_\Omega\left[F_{yy}(x)\psi(x)\cdot\psi(x) + 2 F_{yz}(x)\psi(x)\cdot\nabla\psi(x) + \int F_{zz}(x)z\cdot z \dv{}\nu_x(z)\right]\dx \nonumber \\
\leq &  \int_\Omega F_{yy}(x)\psi(x)\cdot\psi(x) + 2 F_{yz}(x)\psi(x)\cdot\nabla\psi(x)\dx \notag\\
&\quad+\int_\Omega \int F_{zz}(x)z\cdot z\mathbbm{1}_{B(0,m)}(z) \dv{}\nu_x(z)\dx + \eps.
\end{align}
Next, consider the integrand $H\colon\overline{\Omega}\times\R^{N\times d}\rightarrow\R$ given by
\begin{equation*}
H(x,z):=F_{zz}(x)z\cdot z\mathbbm{1}_{B(0,m)}(z).
\end{equation*}
Note that $\mathbbm{1}_{B(0,m)}(z)$ is lower semicontinuous as the indicator function of the open set $B(0,m)$. Hence, $H(x,\cdot)$ is lower semicontinuous for every $x\in\overline{\Omega}$ and, since $\nabla\psi_k$ generates the Young measure $\nu_x$,
\begin{equation*}\label{D2FbyFTYM}
\underset{\Omega}{\int}\int F_{zz}(x)z\cdot z\mathbbm{1}_{B(0,m)}(z)\dv{}\nu_x(z)\dx\leq \liminf_{k\to\infty}\underset{|\nabla\psi_k|<m}{\int}F_{zz}(x)\nabla\psi_k\cdot\nabla\psi_k\dx.
\end{equation*}
Also, $\psi_k\to\psi$ strongly in $\LL^2(\Omega,\R^N)$ and $(\nabla\psi_k)$ is bounded in $\LL^2(\Omega,\R^{N\times d})$, so that the family $(F_{yy}(x)\psi_k\cdot\psi_k + 2 F_{yz}(x)\psi_k\cdot\nabla\psi_k)_k$ is equiintegrable and, by Young measure representation,
\[
 \int_\Omega\left[F_{yy}(x)\psi\cdot\psi + 2 F_{yz}(x)\psi\cdot\nabla\psi\right]\dx=\lim_{k\to\infty} \int_\Omega\left[F_{yy}(x)\psi_k\cdot\psi_k + 2 F_{yz}(x)\psi_k\cdot\nabla\psi_k\right]\dx.
\]
Combining the last two equations with \eqref{eq:pluseps}, we obtain that for all $m\geq m_\eps$,
\begin{align}\label{eq:YMagainstseq}
&\int_\Omega\left[F_{yy}(x)\psi\cdot\psi + 2 F_{yz}(x)\psi\cdot\nabla\psi + \int F_{zz}(x)z\cdot z \dv{}\nu_x(z)\right]\dx \nonumber \\
& \quad \leq \liminf_{k\to\infty}\int_{\{|\nabla\psi_k|<m\}} \left[F_{yy}(x)\psi_k\cdot\psi_k + 2 F_{yz}(x)\psi_k\cdot\nabla\psi_k + F_{zz}(x)\nabla\psi_k\cdot \nabla\psi_k\right]\dx \nonumber \\
&\qquad + \lim_{k\to\infty}\int_{\{|\nabla\psi_k|\geq m\}} \left[F_{yy}(x)\psi_k\cdot\psi_k + 2 F_{yz}(x)\psi_k\cdot\nabla\psi_k\right] \dx + \eps \nonumber \\
&\quad \leq \liminf_{k\to\infty}\int_{\{|\nabla\psi_k|<m\}} \left[F_{yy}(x)\psi_k\cdot\psi_k + 2 F_{yz}(x)\psi_k\cdot\nabla\psi_k + F_{zz}(x)\nabla\psi_k\cdot \nabla\psi_k\right]\dx +2\eps.
\end{align}
Note that the last inequality follows from the fact that $(\nabla\psi_k)$ is measure-tight, the equiintegrability of the sequence  $(F_{yy}(x)\psi_k\cdot\psi_k + 2 F_{yz}(x)\psi_k\cdot\nabla\psi_k)_k$,  and by choosing $m_\eps$ larger if necessary.
\vspace{0.2cm}

\noindent We now claim that
\begin{align}\label{eq:lastbeforemain1}
&\frac12 \liminf_{k\to\infty} \underset{\{|\nabla\psi_k|<m\}}{\int} \left[F_{yy}(x)\psi_k\cdot\psi_k + 2 F_{yz}(x)\psi_k\cdot\nabla\psi_k + F_{zz}(x)\nabla\psi_k\cdot \nabla\psi_k\right]\dx \nonumber \\
& \hspace{5cm} = \liminf_{k\to\infty}\underset{\{|\nabla\psi_k|< m\}\cap\{|\nabla b_k|< m\}}{\int} f_k(x)\dx.
\end{align}
Then, by \eqref{eq:fkminuseps}, \eqref{eq:YMagainstseq} and letting $\eps\to0$ (the dependence on $m_\eps$ will have been removed), we conclude that,
\begin{align*}
&\frac12 \int_\Omega\left[F_{yy}(x)\psi(x)\cdot\psi(x) + 2 F_{yz}(x)\psi(x)\cdot\nabla\psi(x) + \int F_{zz}(x)z\cdot z \dv{}\nu_x(z)\right]\dx\\
\leq &\liminf_{k\to\infty} \int_\Omega f_k(x)\dx,
\end{align*}
concluding the proof of Claim 4. 
\vspace{0.2cm}

\noindent
To prove \eqref{eq:lastbeforemain1}, we introduce the notation
\[
L(ty,tz)[y,z]=L(x,u_0(x)+ty,\nabla u_0(x) + tz)[(y,z),(y,z)]
\]
and note that
\[
\underset{\{|\nabla\psi_k|< m\}\cap\{|\nabla b_k|< m\}}{\int} f_k(x) \dx = I^k_1 + I^k_2 + I^k_3 + I^k_4,
\]
where
\begin{align*}
I^k_1 & = \int_\Omega \mathbbm{1}_{\{|\nabla\psi_k|< m\}\cap\{|\nabla b_k|< m\}} \int_0^1 (1-t) \left[L(t\alpha_k\psi_k, t\alpha_k\nabla\psi_k) - L(0,0)\right][\psi_k,\nabla\psi_k] \dt\dx; \\
I^k_2 & = \frac12 \int_{\{|\nabla\psi_k|<m\}} L(0,0)[\psi_k,\nabla\psi_k] \dx;\\
I^k_3 & = -\frac12 \int_{\{|\nabla\psi_k|<m\}} L(0,0)[\psi_k,\nabla\psi_k] \left(1-\mathbbm{1}_{\{|\nabla b_k|<m\}}\right)\dx; \\
I^k_4 & = - \int_\Omega \mathbbm{1}_{\{|\nabla\psi_k|< m\}\cap\{|\nabla b_k|< m\}}\int_0^1 (1-t) L(t\alpha_k b_k,t\alpha_k\nabla b_k)[b_k,\nabla b_k]\dt\dx. \nonumber
\end{align*}
The term $I^k_2$ is precisely the one appearing on the left-hand side of \eqref{eq:lastbeforemain1} and it thus suffices to prove that $I^k_1$, $I^k_3$ and $I^k_4$ all converge to $0$ as $k\to\infty$.
It is clear, by the Dominated Convergence Theorem, that since $\alpha_k\rightarrow 0$, $I^k_1\to0$ as $k\to\infty$. Regarding the term $I^k_4$, note that the sequence of functions
\begin{equation*}
L(t\alpha_k b_k,t\alpha_k\nabla b_k)[\nabla b_k,\nabla b_k]\mathbbm{1}_{\{|\nabla b_k|<m \}\cap\{ |\nabla \psi_k|<m \}}
\end{equation*}
is bounded in $\LL^\infty(\Omega)$ for all $t\in[0,1]$ (recall  $\alpha_k b_k\to0$ uniformly) and, therefore, it is equiintegrable. In addition, this sequence converges to $0$ in measure because $\nabla b_k\rightarrow 0$ in measure and, by [H0], all partial derivatives of second order of $F$ are continuous. These two facts imply, by Vitali's Convergence Theorem, that $I^k_4\to 0$ as $k\to\infty$.
\vspace{0.2cm}

\noindent
Furthermore, given that $\nabla b_k\rightarrow 0$ in measure and $(|\psi_k|)$ is equiintegrable, we also have that 
\[
|I^k_3|\leq cm\int_\Omega |\psi_k|\left(1- \mathbbm{1}_{\{|\nabla b_k|<m \}}   \right)\dx\rightarrow 0,\mbox{ as }k\to\infty. 
\]
Hence, \eqref{eq:lastbeforemain1} follows and, consequently, Claim 4.  
\vspace{0.2cm}

\noindent{\bf Step 5.} In this last step, we show that the inequality obtained in Claim 4 together with the strict positivity of the second variation, leads to a contradiction. 
\vspace{0.2cm}

\noindent
By Lemma \ref{lemma:F''QC} (see also Remark \ref{remark:F''QC}), for any $x\in\Omega$, the function $H(x,z) = F_{zz}(x)z\cdot z - 2c_0|z - \bar{\nu}_{x}|^2$ is quasiconvex at every $z\in\R^{N\times d}$, where $\bar{\nu}_x = \nabla\psi(x)$ a.e.~in $\Omega$. In particular, since $(\nu_x)_x$ is a gradient Young measure and $H(x,\cdot)$ has quadratic growth, Jensen's inequality from the characterisation of gradient Young measures\footnote{See, for example, \cite{KPCGYM} or \cite[Lemma 5.1]{RindlerCV}.} implies that for a.e. $x\in\Omega$,
\begin{equation}\label{eq:fzzqc}
\frac12 F_{zz}(x)\nabla\psi(x)\cdot \nabla\psi(x) + c_0\int_{\R^{N\times d}} |z -\nabla\psi(x)|^2\dv\nu_x(z) \leq  \frac12 \int_{\R^{N\times d}} F_{zz}(x) z\cdot z \dv\nu_x(z).
\end{equation}
On the other hand, in Step 4 we showed that
\[
\int_\Omega\left[F_{yy}(x)\psi\cdot\psi + 2 F_{yz}(x)\psi\cdot\nabla\psi + \int F_{zz}(x)z\cdot z \,{\rm d}\nu_x(z)\right]\dx\leq 0
\]
and, by \eqref{eq:fzzqc}, we may hence deduce that
\begin{align*}
&\frac12 \int_\Omega\left[F_{yy}(x)\psi\cdot\psi + 2 F_{yz}(x)\psi\cdot\nabla\psi + F_{zz}(x)\nabla\psi\cdot\nabla\psi\right]\dx \\
& \quad+ c_0\int_\Omega\int |z-\nabla\psi|^2\,{\rm d}\nu_x(z)\dx\leq  0.
\end{align*}
However, the second variation has been assumed strongly positive (see assumption $(II)$ in Theorem \ref{thm:sufficiency}), so that the above inequality implies
\[
\frac{c_0}{2}\int_\Omega |\nabla\psi|^2\dx +c_0 \int_\Omega\int |z-\nabla\psi|^2\,{\rm d}\nu_x(z)\dx\leq 0,
\]
i.e. $\nabla\psi=0$, $\psi=0$ (by Poincar\'e's inequality) and $\nu_x = \delta_0$ a.e. in $\Omega$. In particular, we infer that $\nabla\psi_k\rightharpoonup 0$ in $\LL^2(\Omega,\R^N)$ and that $\nabla\psi_k\rightarrow 0$ in measure (since the generated measure is an elementary Young measure, see \cite[Lemma 4.12]{RindlerCV}).
\vspace{0.2cm}

\noindent
To conclude the proof, we first note that 
\begin{equation*}
\alpha_k^\frac{p-2}{p}\psi_k= \beta_k\alpha_k^{-\frac{2}{p}}\eta_k\psi_k.
\end{equation*}
Hence, for a subsequence that we do not relabel, we may use (\ref{betapbyalpha2bded}) to further deduce that $\alpha_k^{\frac{p-2}{p}}\psi_k \rightharpoonup 0$ in $\W^{1,p}(\Omega,\R^N)$. Then, it is clear that
\begin{itemize}
\item $\alpha_k^{-1}S(\alpha_k\psi_k)\longrightarrow 0$ in $\LL^2$ and
\item $\alpha_k^{-1}S(\alpha_k\nabla \psi_k)$ is bounded in $\LL^2$.
\end{itemize}
Whereby, we may apply Proposition \ref{global_garding} with $h_k=\alpha_k\psi_k=\varphi_k$ and $a_k=\alpha_k$ to estimate
\begin{align*}
\frac{c_0}{2}=&\liminf_{k\rightarrow\infty}\frac{c_0}{2}\underset{\Omega}{\int}|\nabla\psi_k|^2\dx\notag\\
\leq&\liminf_{k\rightarrow\infty}\frac{c_0}{2}\underset{\Omega}{\int}\left(|\nabla\psi_k|^2+\alpha_k^{p-2}|\nabla\psi_k|^p\right)\dx\notag\\
\leq & \liminf_{k\rightarrow\infty}\alpha_k^{-2}G(x,\varphi_k,\nabla\varphi_k)\dx\leq 0.
\end{align*}
\noindent
Since $c_0>0$, this is a contradiction and the proof of Theorem \ref{thm:sufficiency} is complete. 

\end{proof}

\section{The case of domains locally diffeomorphic to a cone}\label{sec:AppendDiffeom}

\noindent In this section we establish the proof of Theorem \ref{zhanglemma} for the general case in which $\Omega$ is locally diffeomorphic to a cone. 

\begin{proof}[Proof of Theorem \ref{zhanglemma} (general case)]
Observe first that, by compactness of $\partial\Omega$, there are a finite set of points $\{y_1,...,y_M\}\subseteq \partial\Omega$ and radii ${r}_1,...,{r}_M>0$ such that 
\begin{equation*}
\partial\Omega\subseteq\bigcup_{j=1}^M B\left(y_j,\frac{{r}_j}{2}\right)
\end{equation*}
and for which there exist cones $ \mathcal{C}_1,..., \mathcal{C}_M$ and orientation-preserving diffeomorphisms  \linebreak $g_j\colon\overline{\Omega(y_j,{r}_j)}\rightarrow \mathcal{D}_j\subseteq \overline{\mathcal{C}_j}$ with $g_j(\partial\Omega\cap \overline{B(y_j,{r}_j)})\subseteq \partial \mathcal{C}_j$.  
Hence, for $x_*\in\Gamma_N$ arbitrary, $x_*\in B\left(y_j,\frac{{r}_j}{2}\right)$ for some $1\leq j\leq M$.

Assume temporarily that, for any $\delta>0$, there exist some $R_0>0$ independent of $x_*$, a conical boundary region $D_*$ and a diffeomorphism $\Phi\colon D_*\mapsto \overline{\Omega(x_j,r_j)}$, such that
\begin{eqnarray}
&\Phi(x_*)=x_*;\\
&\left\|\Phi-\mathrm{Id}_{D_*}\right\|_{{\rm C}^1(D_*\cap B(x_*,R_0),\R^d)}+\|\det \nabla\Phi-1\|_{{\rm C}^0(D_*\cap B(x_*,R_0))}<\delta;\label{eqhomeom}\\
&B(x_*,R_1)\subseteq \Phi(B(x_*,R_0))\label{eqhomeom2},
\end{eqnarray}
where $R_1>0$ is a positive radius that does not depend on $x_*$.

Under this assumption, for a fixed $\varepsilon>0$, we use Lemma \ref{LemmagrowthG} (b)-(c) to find $0<\delta<1$ and $R_\varepsilon>0$ such that the inequality
\begin{align}
&\left|G(x_0,0,z\, J)-c_0|S(z\, J)|^2-\left(G(x,y,z)|\det J|-{c_0}|S(z)|^2|\det J|\right) \right|\notag\\
\leq& \,c|y|^2|\det J|+ \varepsilon |S(z)|^2|\det J|\label{mainestimate}
\end{align}
is satisfied whenever $x\in \Omega(x_0,R_\varepsilon)$ and $|y|+|J-\Irm_d|<\delta$. For such a $\delta>0$, we may take $R_0$ such that ($\ref{eqhomeom}$)-\eqref{eqhomeom2} hold. We further assume that $\left|\left|\Phi\right|\right|_{{\rm C}^1(D_*\cap B(x_*,R_0),\R^d)}\leq 2$ and we set 
\[R:=\frac{1}{4}\min\{R_\varepsilon,R_1, {r}_j,\mathrm{diam}(\Omega)\st 1\leq j\leq M\}.
\]
In order to show that $u_0$ indeed satisfies \eqref{eq:minimality} of Theorem \ref{zhanglemma}, we consider the following cases. 
\\
\textit{Case 1.} If $\overline{\Omega(x_0,R)}\cap \Gamma_N\neq\emptyset$, we choose $x_* \in \overline{\Omega(x_0,R)}\cap \Gamma_N$ such that $x_*$ is a  vertex of the cone into which, by assumption, $\Omega(x_0,R)$ can be mapped under a diffeomorphism. 
Whereby, $\Omega(x_0,R)\subseteq\Omega(x_*,2R)$. Furthermore, if $x_*\in \Omega(y_J,\frac{{r_J}}{2})$,  then $\Omega(x_*,2R)\subseteq \Omega(y_J,\frac{{r_J}}{2})$.

We now consider $\Phi$ as above. Then, for any  $\xi\in \Phi^{-1}(\Omega(x_0,R))$, we have $|\Phi(\xi)-x_0|<R$ and, consequently, $|\Phi(\xi)-x_*|<R_1$. Hence, by \eqref{eqhomeom2}, $|\xi-x_*|<R_0$. Then, if $\varphi\in\overline{\mathrm{Var}(\Omega(x_0,R),\R^N)^{\W^{1,p}}}$ satisfies
\begin{equation}
\|\varphi\|_{\LL^\infty(\Omega(x_0,R),\R^N)}<\delta
\end{equation}
with $\delta>0$ as above, in virtue of \eqref{eqhomeom} we can substitute $x=\Phi(\xi)$, $y=\varphi(\Phi(\xi))$,  $z:=\nabla\varphi(\Phi(\xi))$ and $J:=\nabla\Phi(\xi)$ in \eqref{mainestimate} to obtain, after integrating over $\Phi^{-1}(\Omega(x_0,R))$,

\begin{align}
&\int_{\Phi^{-1}(\Omega(x_0,R))}G(x_0,0,\nabla\varphi(\Phi(\xi)) \nabla\Phi(\xi))\dv \xi\notag\\
&\, -c_0\int_{\Phi^{-1}(\Omega(x_0,R))}|S(\nabla\varphi(\Phi(\xi)) \nabla\Phi(\xi))|^2\dv\xi\notag\\
&-\int_{\Phi^{-1}(\Omega(x_0,R))}G(\Phi(\xi),\varphi(\Phi(\xi)),\nabla\varphi(\Phi(\xi)))|\det \nabla\Phi(\xi)|\dv \xi\notag\\
&+ c_0\int_{\Phi^{-1}(\Omega(x_0,R))}\left|S\left(\nabla\varphi(\Phi(\xi))\right)\right|^2|\det \nabla\Phi(\xi)|\dv \xi\notag\\
\leq & \,c \int_{\Phi^{-1}(\Omega(x_0,R))}\left|\varphi(\Phi(\xi))\right|^2|\det \nabla\Phi(\xi)|\dv\xi\notag\\
&+   \varepsilon\int_{\Phi^{-1}(\Omega(x_0,R))}|S(\nabla\varphi(\Phi(\xi))|^2|\det \nabla\Phi(\xi)|\dv \xi.\label{goodineq}
\end{align}
We are interested in using the quasiconvexity at the boundary in order to simplify the above expression. With this aim, we define $\tilde{\varphi}\colon \Phi^{-1}(\Omega(x_0,R)) \rightarrow\R^N$ as
\begin{equation*}
\tilde{\varphi}(\xi):=\varphi\circ\Phi(\xi).
\end{equation*}

We wish to establish that $\tilde{\varphi}$ is a suitable test function for the quasiconvexity at the boundary condition that we have imposed. Indeed, by construction, $D:= \Phi^{-1}(\Omega(x_0,R))$ is a standard conical-boundary region and $\tilde{\varphi}(x)=0$ for all $x\in\partial D\cap\, \mathrm{int}\mathcal{C}$, so that the quasiconvexity at the boundary can be applied at $x_*$ with this test function  by Remark \ref{remark:p-quasiconvexity} and Lemma \ref{lemma:qcindependent}. 

Noting that $\nabla\tilde{\varphi}(\xi)= \nabla\varphi(\Phi(\xi)) \nabla\Phi(\xi)$, the quasiconvexity at the boundary condition reads 
\begin{align}
0 \leq & \int_{\Phi^{-1}(\Omega(x_0,R))}G(x_0,0,\nabla(\varphi\circ\Phi)(\xi) )\dv\xi
 -{c_0}\int_{\Phi^{-1}(\Omega(x_0,R))}|S(\nabla(\varphi\circ\Phi)(\xi)) |^2\dv\xi.
\label{QCatbdrywithGII}
\end{align}
From inequalities \eqref{goodineq} and  \eqref{QCatbdrywithGII}, we infer that
\begin{align}
&-\int_{\Phi^{-1}(\Omega(x_0,R))}G(\Phi(\xi),\varphi\circ\Phi(\xi),\nabla\varphi(\Phi(\xi))|\det \nabla\Phi(\xi)|\dv\xi\notag\\
&+ c_0\int_{\Phi^{-1}(\Omega(x_0,R))}|S\left(\nabla\varphi(\Phi(\xi))\right)|^2|\det \nabla\Phi(\xi)|\dv\xi\notag\\
\leq & \,\varepsilon\int_{\Phi^{-1}(\Omega(x_0,R))}|S(\nabla\varphi(\Phi(\xi)))|^2|\det \nabla\Phi(\xi)|\dv\xi\notag\\
&+c\int_{\Phi^{-1}(\Omega(x_0,R))}|\phi(\Phi(\xi))|^2|\det \nabla\Phi(\xi)|\dv\xi.\notag
\end{align}

Applying the change of variables $x=\Phi(\xi)$, this leads to
\begin{align}
&-\int_{\Omega(x_0,R)}G(x,\varphi(x),\nabla\varphi(x))+ c_0|S\left(\nabla\varphi(x)\right)|^2\dx\notag\\
\leq&  \,\varepsilon\int_{\Omega(x_0,R)}|S(\nabla\varphi(x))|^2\dx+c\int_{\Omega(x_0,R)}|\varphi(x)|^2\dx.\label{bestineq}
\end{align}
Also, since $\varphi\in\overline{\mathrm{Var}(\Omega(x_0,R),\R^N)^{\W^{1,p}}}$ and $u_0$ is an $F$-extremal, 
\begin{equation*}
\int_{\Omega(x_0,R)}F_y(x,u_0(x),\nabla u_0(x))\cdot\varphi(x)+ F_z(x,u_0(x),\nabla u_0(x))\cdot \nabla\varphi(x)\dx=0.
\end{equation*}
This, together with \eqref{bestineq}, imply for $\varepsilon=\frac{c_0}{2}$ that 
\begin{align*}
&\int_{\Omega(x_0,R)}\left(F(x,u_0(x)+\varphi(x),\nabla u_0(x)+\nabla\varphi(x))-F(x,u_0(x),\nabla u_0(x))\right)\dx\notag\\
& -c_0\int_{\Omega(x_0,R)}|S\left(\nabla\varphi(x)\right)|^2\dx+ c\int_{\Omega(x_0,R)}|\varphi(x)|^2\dx \notag\\
\geq & \,-\frac{c_0}{2}\int_{\Omega(x_0,R)}|S\left(\nabla\varphi(x)\right)|^2\dx,
\end{align*}
which gives the desired inequality after adding $c_0\int_{\Omega(x_0,R)}|S\left(\nabla\varphi(x)\right)|^2\dx$ to both sides of the above expression. 
This concludes the proof of ($\ast$) for $x_0$ in a neighbourhood of $\Gamma_N$. 
\\
\textit{Case 2.} To prove ($\ast$) if $x_0\in\overline{\Omega}$ is not in the neighbourhood of radius $R$ of $\Gamma_N$, a simpler version of the above proof works, since we can then use that the standard quasiconvexity holds in $\overline{\Omega}$ and take $\Phi$ as the identity diffeomorphism in the above proof, given that there is no need, in this case, to transform the boundary into a subset of a cone. All other calculations follow in the exact same way.
\\
To conclude the proof of the theorem, it remains to show that the diffeomorphism $\Phi$ can be constructed with the required properties. We emphasize that the role of $\Phi$ is to locally ``blow-up" the boundary of $\Omega$ by smoothly mapping it into the corresponding cone. 

We observe that, for $x_*\in\Gamma_N$ arbitrary, $x_*\in B(y_j,r_j)$ for some $1\leq j\leq M$. Recall that $\mathcal{D}_j=g_j[\overline{\Omega(y_j,{r}_j)}]$ and write
\[
A:=\nabla g_j(x_*)\in\R^{d\times d} \mbox{ \,\,\,and \,\,\,}D_*:=A^{-1}\mathcal{D}_j-A^{-1}  g_j(x_*)+x_*\subseteq\R^d.
\] 

Consider the diffeomorphism $\Phi\colon D_*\rightarrow \overline{\Omega(x_j,{r}_j)}$ given by
\begin{equation}\label{defPhi1}
\Phi(\xi):=g_j^{-1}\left( A(\xi-x_*)+g_j(x_*)\right).
\end{equation}  
Clearly, 
\begin{equation*}
\Phi^{-1}(x)=A^{-1}\cdot \left(g_j(x)-g_j(x_*)\right)+x_*
\end{equation*}
and, hence, $\Phi$ is well defined. 
We then observe that, for $\xi\in D_*$, 
\begin{align}
&|\Phi(\xi)-\xi|+|\nabla \Phi(\xi)-\Irm_d|\notag\\
=&|g_j^{-1}(A(\xi-x_*)+g_j(x_*))-g_j^{-1}(g_j(\xi))|\notag\\
+&\left| \left[\nabla g_j^{-1} (A(\xi-x_*)+g_j(x_*) )-\nabla g_j(x_*)^{-1}\right]\nabla g_j(x_*)   \right|\notag\\
=&|g_j^{-1}(A(\xi-x_*)+g_j(x_*))-g_j^{-1}(g_j(\xi))|\notag\\
+&\left| \left[\nabla g_j^{-1} (A(\xi-x_*)+g_j(x_*) )-\nabla g_j^{-1}(g_j(x_*))\right]\nabla g_j(x_*)   \right|\notag\\
\leq& c\omega_0^j(|A||\xi-x_*|+|g_j(x_*)-g_j(\xi)|)+c|A|\omega_1^j(|A||\xi-x_*|),
\end{align}
where $\omega_0^j$ and $\omega_1^j$ are moduli of continuity of $g_j^{-1}$ and $\nabla g_j^{-1}$ respectively.

Since $g_j$ is continuous over the compact set $\overline{\Omega(y_j,{r}_j)}$ and the set $\{g_1,...,g_M\}$ is finite, 
this implies that, given $\delta>0$, there exists some $R_0>0$, that does not depend on $x_*$, such that (\ref{eqhomeom}) is satisfied. 

By uniform continuity and the fact that $\Phi(x_*)=x_*$, we can further find $R_1>0$, independent of $x_*$, such that (\ref{eqhomeom2}) also holds. 

\end{proof}


 \section{Appendix: Proof of the technical lemma}\label{Appendix}
 
For the ease of the reader, we restate Lemma \ref{LemmagrowthG} which we prove in this section. We remark that its proof is motivated by the truncation strategy originated in \cite{AcFus}. 

\begin{customthm}{4.4}
Let $\Omega\subset\R^d$ be a Lipschitz domain. Assume further that $F\colon\overline{\Omega}\times \R^N\times \R^{N\times d}\rightarrow\R$ is a continuous integrand, $u_0\in {\rm C}^1(\overline{\Omega},\R^N)$ and that [H0], [H1] and [UC] hold. 
Define the function $G\colon\overline{\Omega}\times \R^N\times \R^{N\times d}\rightarrow\R$  by 
\begin{align*}
G(x,y,z):=&F(x,u_0(x)+y,\nabla u_0(x)+z)-F(x) -F_y(x)\cdot y-F_z(x)\cdot z\notag\\
=&\int_0^1 (1-t)L(x, u_0(x)+ty,\nabla u_0(x)+tz)[(y,z),(y,z)],
\end{align*}
where the bilinear form $L(x,v,w)$ is given by 
\begin{align*}
L(x,v,w)[(y,z),(\hat{y},\hat{z})]:= &F_{yy}(x,v,w)y\cdot \hat y+ F_{yz}(x,v,w)y\cdot z+F_{yz}(x,v,w)\hat{y}\cdot \hat{z} \\&+F_{zz}(x,v,w)z\cdot \hat{z}.
\end{align*}

The following hold:

\begin{itemize}
\item[(a)] for each $x\in\overline{\Omega}$, $y,\hat y\in\R^N$ and $z,\hat z\in\R^{N\times d}$ and for some locally bounded function $C(y,\hat y)$ on $\R^{2N}$, 
\[
|G(x,y,z) - G(x,\hat y,\hat z)|\leq C(y,\hat y)\left(A_{p-1}(y,z,\hat y,\hat z)|z-\hat z| + A_p(y,z,\hat y,\hat z)|y-\hat y|\right),
\]
where
\[
A_p(y,z,\hat y,\hat z) = |y|+|\hat y|+|z|+|\hat z|+|z|^p+|\hat z|^p.
\]
In particular, for some locally bounded function $C(y)$,
\[
|G(x,y,z)|\leq  C(y)\left(|S(y)|^2+|S(z)|^2\right);
\] 
\item[(b)]  for every $\varepsilon>0$ there exist $R=R(\varepsilon)>0$ and $\delta=\delta(\varepsilon)>0$ such that, for all $x_0,x\in\overline{\Omega}$, $z\in\R^{N\times d}$ and $J\in\R^{d\times d}$, if $|x-x_0|<R$,  $w=z\, J$ and $|y|+|J-\mathrm{I}_d|<\delta$, then
\begin{equation*}
|G(x_0,0,w)-G(x,y,z)|\det J||<c|y|^2|\det J|+\varepsilon |S(z)|^2|\det J|
\end{equation*}
for some constant $c>0$ where $\mathrm{I}_d$ denotes the $d\times d$ identity matrix;
\item[(c)] for every $\varepsilon>0$ there exists $\delta=\delta(\varepsilon)>0$ such that, for  all $z\in\R^{N\times d}$ and $J\in\R^{d\times d}$, if  $w=z\, J$ and $|J-\mathrm{I}_d|<\delta$, then 
\begin{align*}
c_0\left||S(z\, J)|^2-|S(z)|^2|\det J|\right|< & \,{\varepsilon} |S(z)|^2.
\end{align*}
 \end{itemize}
\end{customthm}

\begin{proof}

To prove (a), by the triangle inequality we estimate
\begin{equation*}
|G(x,y,z)-G(x,\hat y, \hat z)|\leq |G(x,y,z)-G(x,\hat y, z)|+|G(x,\hat y, z) - G(x, \hat y, \hat z)|=:\mathrm{I+II}.
\end{equation*}
We next use a truncation strategy and consider two cases.

\textit{Case 1:} $|y|+|\hat y|+|z|+|\hat z|\leq 1$.
The approach that we follow is based on the quadratic behaviour of $G$ for small values of $|y|+|\hat y|+|z|+|\hat z|$. 
\begin{align*}
\mathrm{I}  
\leq& \left|\int_0^1  \left(F_y(x,u_0(x)+\hat y+t(y-\hat y),\nabla u_0(x)+z)-F_y (x,u_0(x) ,\nabla u_0(x)+z)  \right)\cdot(y-\hat y)\dt \right|\\
&+\left|\int_0^1  \left(F_y(x,u_0(x),\nabla u_0(x)+z)-F_y (x,u_0(x),\nabla u_0(x))  \right)\cdot(y-\hat y)\dt \right|\\
\leq& \int_0^1\int_0^1 \left|F_{yy}(x,u_0(x)+st(y-\hat y),\nabla u_0(x)+z)(t(y-\hat y))\cdot (y-\hat y)\right|\dv s\dt\\
&+\int_0^1\int_0^1 \left|F_{zy}(x,u_0(x),\nabla u_0(x)+sz)z\cdot (y-\hat y)\right|\dv s\dt\\
\leq& \,c\left( |y|+|\hat y|+|z|+|\hat z|   \right)|y-\hat y|.
\end{align*}
Moreover, 
\begin{align*}
\mathrm{II}  
\leq& \int_0^1\int_0^1 \left|F_{yz}(x,u_0(x)+s \hat y,\nabla u_0(x)+\hat z+t(z-\hat z))\hat y\cdot (z-\hat z)\right|\dv s\dt\\
&+\int_0^1\int_0^1 \left|F_{zz}(x,u_0(x),\nabla u_0(x)+s\hat z) \hat z\cdot (z-\hat z)\right|\dv s\dt\\
\leq& \,c\left( |y|+|\hat y|+|z|+|\hat z|   \right)|z-\hat z|.
\end{align*}
We emphasize that the last inequality in each of the two estimates above relies only on condition [H0] and the fact that $u_0$ and $\nabla u_0$ are uniformly bounded. 

\textit{Case 2:} $|y|+|\hat y|+|z|+|\hat z|> 1$. 
By the triangle inequality, the Fundamental Theorem of Calculus and [H1] (b), we obtain that
\begin{align*}
\mathrm{I}\leq &  \int_0^1 \left| F_y(x,u_0(x)+\hat y+ t(y-\hat y), \nabla u_0(x)+z)\cdot(y-\hat y)\right|\dt +\left|F_y(x)\cdot(y-\hat y)\right|\\
\leq & \, C(y,\hat y)(|y|+|\hat y|+|z|+|\hat z|+|z|^p)||y-\hat y|.
\end{align*}
Similarly, we can deduce that 
\begin{align*}
\mathrm{II}\leq &  \int_0^1 \left| F_z(x,u_0(x)+\hat y, \nabla u_0(x)+z+t(z-\hat z))\cdot(z-\hat z)\right| \dt +\left|F_z(x)\cdot(z-\hat z)\right|\\
\leq & C(y,\hat y)(|y|+|\hat y|+|z|+|\hat z|+|z|^{p-1}+|\hat z|^{p-1})||z-\hat z|.
\end{align*}
This concludes the proof of (a), after using the definition of $A_p$.

To prove (b), fix $\varepsilon>0$ and estimate as
\begin{align*}
|G(x_0,0,w)-G(x,y,z)|\leq & |G(x_0,0,w)-G(x_0,0,z)|+|G(x_0,0,z)-G(x,0,z)|\\ 
&+|G(x,0,z)-G(x,y,z)|
=:\mathrm{I}+\mathrm{II}+\mathrm{III}. 
\end{align*}

To estimate term $\mathrm{I}$, note that $w = zJ$ and we may assume that $|J -{\rm I}_d|<1$ so that $J\leq c(d)$. Then, setting $y = \hat y = 0$ in part (a) we deduce that, for $\delta$ small enough, 
\begin{align*}
\mathrm{I} = |G(x_0,0,w)-G(x_0,0,z)| & \leq c \big(|z| + |w| + |z|^{p-1}+|w|^{p-1}\big)|z-w|\\
& \leq c  \big(|z|+ |z|^{p-1}\big)|J-{\rm I}_d||z|\\
& < {\varepsilon} |S(z)|^2 .
\end{align*}
To estimate term $\mathrm{III}$, we may take $|y|<1$ and use part (a) again with $\hat z = z$ and $\hat y = 0$, to find that
\begin{align*}
\mathrm{III} = |G(x,0,z)-G(x,y,z)| & \leq c \big(|y| + |z| + |z|^{p}\big)|y|\\
& \leq c \big(|y|^2 + |y||z|+|y| |z|^{p}\big)\\
& \leq c \big((1+C_{\varepsilon})|y|^2 + \frac{\varepsilon}{c} |S(z)|^2\big)\\
& = c|y|^2 + {\varepsilon}|S(z)|^2 ,
\end{align*}
where the last inequality follows by setting $|y|<\delta < \varepsilon/c$ and Young's inequality. The estimate for $\mathrm{II}$ follows a similar strategy to the proof of (a). In particular, for $|z|\leq 1$,
\begin{align*}
|G(x_0,0,z)-G(x,0,z)| \leq & \int_0^1\left|F_{zz}(x,u_0(x),\nabla u_0(x)+tz)-F_{zz}(x_0,u_0(x_0),\nabla u_0(x_0)+tz)\right||z|^2\dt\notag\\
\leq &c\tilde{\omega}_0(|x-x_0|)|z|^2,
\end{align*}
where $\tilde{\omega}_0$ is a modulus of continuity depending, in this case, also on the ${\rm C}^1$ function $u_0$. 

On the other hand, for $|z|>1$, we observe that
\begin{align*}
|G(x_0,0,z)-G(x,0,z)| \leq& |F(x,u_0(x),\nabla u_0(x)+z)-F(x_0,u_0(x_0),\nabla u_0(x_0)+z)|\notag\\
&+|F(x_0,u_0(x_0),\nabla u_0(x_0))-F(x,u_0(x),\nabla u_0(x))|\notag\\
&+\left|F_z(x,u_0(x),\nabla u_0(x))-F_z(x_0,u_0(x_0),\nabla u_0(x_0))\right||z|.
\end{align*}
From the uniform continuity assumption [UC] and the continuity of $F_z$ and $\nabla u_0$, it is clear that there exists an $R=R(\varepsilon)>0$ such that if $|x-x_0|<R$, it holds that
\begin{equation}
\mathrm{II} \leq \frac{\varepsilon}{2}(1+|z|^p)+\frac{\varepsilon}{2}|z|\leq \varepsilon |S(z)|^2.
\end{equation}
The two estimates together imply that there exists an $R=R(\varepsilon)>0$ such that if $|x-x_0|<R$ and $z\in \mathbb{R}^{N\times d}$
\[
\mathrm{II} = |G(x_0,0,z)-G(x,0,z)| < \varepsilon |S(z)|^2.
\]
We thus deduce that, given $\varepsilon > 0$, we can find $R=R(\varepsilon)>0$ and $\delta = \delta(\varepsilon)>0$ such that for all $x_0,x\in \overline{\Omega}$ with $|x - x_0|<R$ and for all $z\in \mathbb{R}^{N\times d}$ and $J\in\mathbb{R}^{d\times d}$ with $|y| + |J - {\rm I}_d|<\delta$
\[
|G(x_0,0,w)-G(x,y,z)| = \mathrm{I} + \mathrm{II} + \mathrm{III} < c |y|^2 + \varepsilon |S(z)|^2.
\]
The conclusion of (b) follows from this and the continuity of the determinant.

Finally, to prove (c) we observe that, for any given $C,\varepsilon>0$ there is a $\delta=\delta(\varepsilon)\in (0,\min\{\frac{\varepsilon}{2},1\})$ such that, if $|J-\mathrm{I}_d|<\delta$ with $J\in\R^{d\times d}$, we can ensure that $|J-\mathrm{I}_d|+|1-|\det(J)||<\frac{\varepsilon}{4C}$. We now estimate, for any $z\in\R^{N\times d}$ and $J\in\R^{d\times d}$ with $|J-\mathrm{I}_d|<\delta$ as above, that
\begin{align}
{c_0}\left||S(z\, J)|^2-|S(z)|^2|\det J|\right|\leq & \,{c_0}\left[\left||S(z\, J)|^2-|S(z)|^2\right|+|S(z)|^2|1-|\det J||\right]\notag\\
&\quad+C|S(z)|^2|1-|\det J|| \label{prelimauxineq}\\
\leq &\, C\left[ |S(z)|^ 2|J-\mathrm{I}_d|+|S(z)|^2|1-|\det J||\right]\notag \\
< & \,{\varepsilon} |S(z)|^2.
\end{align}
We remark that inequality (\ref{prelimauxineq}) follows from the Lipschitz properties of the function $S$, since $|t^p-s^p|\leq c(p)(t^{p-1}+s^{p-1})|t-s|$ for $t,s\geq 0$, $p\geq 1$.  Note also that we are using $|J|\leq C$, since $|J-\mathrm{I}_d|<\delta<1$.  This concludes the proof.

\end{proof}

\section{Acknowledgements}
The authors wish to thank John Ball and Jan Kristensen for useful discussions that contributed to the ideas in this paper. They also thank the anonymous referees whose careful reviews and comments contributed to improve the present work. J.C.C.  is grateful for the financial support granted by the Sofia Kovalevskaia Foundation - Mexican Mathematical Society and for the hospitality of the Mathematics Departments of the University of Augsburg and the University of Sussex, where this project was, respectively, started and concluded.



\end{document}